\documentclass[11pt]{amsart}
\usepackage{palatino}
\usepackage{amsfonts}
\usepackage{amssymb}
\usepackage{amscd}
\usepackage{xypic}
\usepackage[breaklinks,bookmarksopen,bookmarksnumbered]{hyperref}
\usepackage[dvips]{graphicx}

\newcommand{\dblq}{{/\!/}}
\newtheorem{theorem}{Theorem}[section]
\newtheorem{proposition}[theorem]{Proposition}
\newtheorem{corollary}[theorem]{Corollary}
\newtheorem{lemma}[theorem]{Lemma}
\theoremstyle{definition}
\newtheorem{definition}[theorem]{Definition}
\newtheorem{remark}[theorem]{Remark}

\newtheorem{conjecture/question}[theorem]{Conjecture/Question}

\newtheorem{remark/definition}[theorem]{Remark/Definition}
\newtheorem{terminology/notation}[theorem]{Terminology/Notation}

\def\GG{{\textbf G}}

\def\PP{{\textbf P}}
\def\OO{\mathcal{O}}
\def\cN{\mathcal{N}}

\def\cA{\mathcal{A}}
\def\F{\mathcal{F}}
\def\P{\mathcal{P}}

\def\cS{\mathcal{S}}
\def\K{\mathcal{K}}

\def\I{\mathcal{I}}

\def\cM{\mathcal{M}}
\def\cR{\mathcal{R}}
\def\rr{\overline{\mathcal{R}}}

\def\cU{\mathcal{U}}

\def\Pic0{{\rm Pic}^0(X)}

\def\mm{\overline{\mathcal{M}}}
\def\ss{\overline{\mathcal{S}}}
\def\kk{\overline{\mathcal{K}}}

\def\thet{\overline{\Theta}_{\mathrm{null}}}

\begin{document}
\title{Moduli of theta-characteristics via Nikulin surfaces}

\author[G. Farkas]{Gavril Farkas}

\address{Humboldt-Universit\"at zu Berlin, Institut F\"ur Mathematik,  Unter den Linden 6
\hfill \newline\texttt{}
 \indent 10099 Berlin, Germany} \email{{\tt farkas@math.hu-berlin.de}}
\thanks{}

\author[A. Verra]{Alessandro Verra}
\address{Universit\'a Roma Tre, Dipartimento di Matematica, Largo San Leonardo Murialdo \hfill
 \newline \indent 1-00146 Roma, Italy}
 \email{{\tt
verra@mat.unirom3.it}}

\maketitle

The importance of the locus $\K_g:=\{[C]\in \cM_g: C \mbox{ lies on a } K3 \mbox{ surface}\}$
has been recognized for some time. Fundamental results in the theory of algebraic curves
like the Brill-Noether Theorem \cite{Laz}, or Green's Conjecture for generic curves \cite{Vo} have been proved by specialization to a general point $[C]\in \K_g$.
The variety $\K_g$ viewed as a subvariety of $\cM_g$ serves as an obstruction for effective divisors on $\mm_g$ to having small slope
\cite{FP} and thus plays a significant role in determining the cone of effective divisors on $\mm_g$.
\vskip 3pt
The first aim of this paper is to show that at the level of the the \emph{Prym moduli space} $\cR_g$ classifying \'etale double covers of curves of genus $g$, the locus of curves lying on a \emph{Nikulin $K3$ surfaces} plays a similar role. The analogy is far-reaching: Nikulin surfaces furnish an explicit unirational parametrization of $\cR_g$ in small genus, see Theorem \ref{uniruledpar}, just like ordinary $K3$ surfaces do the same for $\cM_g$;
numerous results involving curves on $K3$ surfaces have a Prym-Nikulin analogue, see Theorem \ref{r7}, and even exceptions to uniform statements concerning curves on $K3$ surfaces carry over in this analogy!
\vskip 1pt

Our other aim is to complete the birational classification of the moduli space $\ss_g^+$ of even spin curves of genus $g$.
It is known \cite{F1} that $\ss_g^+$ is of general type when $g\geq 9$. Using Nikulin surfaces we show that $\ss_g^+$ is uniruled for $g\leq 7$, see Theorem \ref{eventheta}, which leaves $\ss_8^+$ as the only case missing from the classification. We prove the following:
\begin{theorem}\label{spin8}
The Kodaira dimension of $\ss_8^+$ is equal to zero.
\end{theorem}
Theorems \ref{spin8} and \ref{eventheta} highlight the fact that the birational type of $\ss_g^+$ is entirely governed by the world of $K3$ surfaces, in the sense  that $\ss_g^+$ is uniruled precisely when a general even spin curve of genus $g$ moves on a special $K3$ surface. This is in contrast to $\mm_g$ which  is known to be uniruled at least for $g\leq 16$, whereas the general curve of genus $g\geq 12$ does not lie on a $K3$ surface.
\vskip 2pt

A \emph{Nikulin  surface} \cite{Ni} is a $K3$ surface $S$ endowed with a non-trivial double cover
$$f: \tilde S \rightarrow S $$
with a branch divisor
$N:=N_1 + \cdots + N_8$
consisting of $8$ disjoint smooth rational curves $N_i\subset S$. Blowing down the $(-1)$-curves $E_i:=f^{-1}(N_i)\subset \tilde{S}$, one obtains a minimal $K3$ surface $\sigma:\tilde{S}\rightarrow Y$, together with an involution $\iota\in \mbox{Aut}(Y)$ having $8$ fixed points corresponding to the images $\sigma(E_i)$ of the exceptional divisors. The class $\mathcal O_S(N)$ is divisible by $2$ in $\mathrm{Pic}(S)$ and we set $e:=\frac{1}{2}\OO_S(N_1+\cdots+N_8)\in \mathrm{Pic}(S).$  Assume that $C\subset S$ is a smooth curve of genus $g$ such that
$C\cdot N_i = 0$ for $i=1, \ldots, 8$. We say that the triple $(S, e, \mathcal O_S(C))$ is a \emph{polarized Nikulin surface of genus $g$} and denote by $\F_g^{\mathfrak{N}}$ the $11$-dimensional moduli space of such objects. Over $\F_g^{\mathfrak{N}}$ we consider the $\PP^g$-bundle
$$\mathcal{P}_g^{\mathfrak{N}}:=\Bigl\{(S, e, C): C\subset S \mbox{ is a smooth curve such that }\  [S, e, \OO_S(C)]\in \F_g^{\mathfrak{N}}\Bigr\},$$
which comes equipped with two maps
$$\xymatrix{
  & \mathcal{P}_g^{\mathfrak{N}} \ar[dl]_{p_g} \ar[dr]^{\chi_g} & \\
   \F_g^{\mathfrak{N}} & & \cR_{g}       \\
                 }$$
where $p_g([S, e, C]):=[S, e, \OO_S(C)]$ and $\chi_g([S, e, C]):=[C, e_C:=e\otimes \OO_C]$.
Since $C\cdot N=0$, it follows that $e_C^{\otimes 2}=\OO_C$. The \'etale double cover induced by $e_C$ is precisely the restriction
$
f_C:=f_ {|\tilde C}: \tilde C \rightarrow C,
$
where $\tilde{C}:=f^{-1}(C)$.
Note that $\mbox{dim}(\mathcal{P}_g^{\mathfrak{N}})=11+g$ and it is natural to ask when is $\chi_g$  dominant and induces a uniruled parametrization of $\cR_g$.
\begin{theorem}\label{uniruledpar}
The general Prym curve $[C, e_C]\in \cR_g$ lies on a Nikulin surface if and only if $g\leq 7$ and $g\neq 6$, that is, the morphism $\chi_g:\mathcal{P}_g^{\mathfrak{N}}\rightarrow \cR_g$ is dominant precisely in this range.
\end{theorem}
In contrast, the general Prym curve $[C, e_C]\in \cR_6$ lies on an Enriques surface \cite{V1} but not on a Nikulin surface. Since $\mathcal{P}_g^{\mathfrak N}$ is a uniruled variety being a $\PP^g$-bundle over $\F_g^{\mathfrak{N}}$, we derive from Theorem \ref{uniruledpar} the following immediate consequence:
\begin{corollary}\label{uniruled2}
The Prym moduli space $\cR_g$ is uniruled for $g\leq 7$.
\end{corollary}
The discussion in Sections 2 and 3 implies the stronger result that $\F_g^{\mathfrak{N}}$ (and thus $\cN_g:=\mbox{Im}(\chi_g)$) is unirational for $g\leq 6$. It was known that $\cR_g$ is rational for $g\leq 4$, see \cite{Do2}, \cite{Ca}, and unirational for $g=5, 6$, see \cite{D}, \cite{ILS}, \cite{V1}, \cite{V2}. Apart from the result in genus $7$ which is new, the significance of Corollary \ref{uniruled2} is that Nikulin surfaces provide an \emph{explicit uniform parametrization} of $\cR_g$ that works for all genera $g\leq 7$.
\vskip 3pt

Before going into a more detailed explanation of our results on $\mathcal{F}_g^{\mathfrak{N}}$, it is instructive to recall Mukai's work on the moduli space $\F_g$ of polarized $K3$ surfaces of genus $g$:
\vskip 3pt

\noindent \underline{Mukai's results \cite{M1}, \cite{M2}, \cite{M3}}:
\vskip 3pt

\noindent $(1)$ A general curve $[C]\in \cM_g$ lies on a $K3$ surface  if and only if $g\leq 11$ and $g\neq 10$, that is, the equality $\K_g=\cM_g$ holds precisely in this range.
\vskip 2pt

\noindent $(2)$  $\cM_{11}$ is birationally isomorphic to the tautological $\PP^{11}$-bundle $\mathcal{P}_{11}$ over the moduli space $\F_{11}$ of polarized $K3$ surfaces of genus $11$. There is a commutative diagram
$$\xymatrix{
\cM_{11} \ar@{<.}[rr]^{\cong}_{q_{11}} \ar[rd] &  & \mathcal{P}_{11} \ar[dl]^{p_{11}}&\\
&                         \F_{11}}$$
with $q_{11}^{-1}([C])=[S, C]$, where $S$ is the unique $K3$ surface containing a general $[C]\in \cM_{11}$.
\vskip 2pt

\noindent $(3)$ The locus $\K_{10}$ is a divisor on $\cM_{10}$ which has the following set-theoretic incarnation:
$$\K_{10}=\bigl\{[C]\in \cM_{10}: \exists L\in W^4_{12}(C) \mbox{ such that } \mu_0(L): \mathrm{Sym}^2 H^0(C, L) \stackrel{\ncong}\longrightarrow H^0(C, L^{\otimes 2})\bigr\}.$$

\noindent $(4)$ There exists a rational variety $X\subset \PP^{13}$ with $K_X=\OO_X(-3)$ and $\mathrm{dim}(X)=5$, such that the general $K3$ surface of genus $10$ appears as a $2$-dimensional linear section of $X$. Such a realization is unique up to the action of $\mathrm{Aut}(X)$ and one has birational isomorphisms:
$$
\F_{10} \stackrel{\cong}\dashrightarrow G\bigl(\PP^{10}, \PP^{13}\bigr)^{\mathrm{ss}}\dblq \mathrm{Aut}(X) \ \ \mbox{ and } \  \
\K_{10} \stackrel{\cong}\dashrightarrow G\bigl(\PP^{9}, \PP^{13}\bigr)^{\mathrm{ss}} \dblq \mathrm{Aut}(X).$$

To this list of well-known results, one could add the following statement from \cite{FP}:
\vskip 2pt

\noindent $(5)$ The closure $\kk_{10}$ of $\K_{10}$ inside $\mm_{10}$ is an extremal point in the effective cone $\mathrm{Eff}(\mm_{10})$; its class
$\kk_{10}\equiv 7\lambda-\delta_0-5\delta_1-9\delta_2-12\delta_3-14\delta_4-\cdots\in \mathrm{Pic}(\mm_{10})$ has minimal slope among all effective divisors on $\mm_{10}$ and provides a counterexample to the Slope Conjecture \cite{HMo}.
\vskip 4pt

Quite remarkably, each of the statements (1)-(5) has a precise Prym-Nikulin  analogue. Theorem \ref{uniruledpar} is the analogue of (1). For the highest genus when the Prym-Nikulin condition is generic, the moduli space acquires a surprising Mori fibre space structure:
\begin{theorem}\label{r7}
The moduli space $\cR_{7}$ is birationally isomorphic to the tautological $\PP^{7}$-bundle $\mathcal{P}_{7}^{\mathfrak{N}}$ and there is a commutative diagram:
$$\xymatrix{
\cR_{7} \ar@{<.}[rr]^{\cong}_{\chi_{7}} \ar[rd] &  & \mathcal{P}_{7}^{\mathfrak{N}} \ar[dl]^{p_{7}}&\\
&                         \F_{7}^{\mathfrak{N}}}$$
Furthermore, $\chi_7^{-1}([C, \eta])=[S, C]$, where the unique Nikulin surface $S$ containing $C$  is given by the base locus of the net of quadrics containing the Prym-canonical embedding
$\phi_{K_C\otimes \eta}:C\rightarrow \PP^5$.
\end{theorem}

Just like in Mukai's work, the genus next to maximal from the point of view of Prym-Nikulin theory, behaves exotically.
\begin{theorem}\label{gen6}
The Prym-Nikulin locus $\cN_6:=\mathrm{Im}(\chi_6)$ is a divisor on $\cR_6$ which can be identified with the ramification locus of the Prym map \
$\mathrm{Pr}_6:\cR_6\rightarrow \cA_5$:
$$\cN_6=\bigl\{[C, \eta]\in \cR_6: \mu_0(K_C\otimes \eta):\mathrm{Sym}^2 H^0(C, K_C\otimes \eta)\stackrel{\ncong}\longrightarrow H^0(C, K_C^{\otimes 2})\bigr\}.$$
\end{theorem}

\noindent Observe that both divisors $\K_{10}$ and $\cN_6$ share the same Koszul-theoretic description. Furthermore, they are both extremal points in their respective effective cones, cf. Proposition \ref{extremal6}. Is there a Prym analogue of the genus $10$ \emph{Mukai $G_2$-variety} $X:=G_2/P\subset \PP^{13}$? The answer to this question is in the affirmative and we outline the construction of a Grassmannian model for $\F_6^{\mathfrak{N}}$ while  referring to Section 3 for details.

\vskip 3pt
Set $V:=\mathbb C^5$ and $U:=\mathbb C^4$ and view $\PP^3=\PP(U)$ as the space of planes inside $\PP(U^{\vee})$. Let us choose a smooth quadric $Q\subset \PP(V)$. The quadratic line complex $W_Q\subset G(2, V)\subset \PP(\wedge^2 V)$ consisting of tangent lines to $Q$ is singular along  the codimension $2$ subvariety $V_Q$ of lines contained in $Q$. One can identify $V_Q$ with the Veronese $3$-fold
$$\nu_2\bigl(\PP^3\bigr)\subset \PP\bigl(\mathrm{Sym}^2(U)\bigr)=\PP(\wedge ^2 V)=\PP^9.$$ The projective tangent bundle $\PP_Q$ of $Q$, viewed as the blow-up of $W_Q$ along $V_Q$, is endowed with a double cover branched along $V_Q$ and induced by the map $$\PP^3\times \PP^3\stackrel{2:1}\longrightarrow \PP\bigl(\mathrm{Sym}^2(U)\bigr), \ \ \mbox{  } (H_1, H_2)\mapsto H_1+H_2.$$ We show in Theorem \ref{param6} that codimension $3$ linear sections of $W_Q$ are Nikulin surfaces of genus $6$ with general moduli. Moreover there is a birational isomorphism
$$\F_6^{\mathfrak{N}}\stackrel{\cong}\dashrightarrow G(7, \wedge^2 V)^{\mathrm{ss}}\dblq \mathrm{Aut}(Q).$$
Taking codimension $4$ linear sections of $W_Q$ one obtains a similar realization of $\cN_6$, which should be viewed as the Prym counterpart of Mukai's construction of $\K_{10}$.
\vskip 4pt

The subvariety $\K_g\subset \cM_g$ is \emph{intrinsic in moduli}, that is, its generic point $[C]$ admits characterizations that involve $C$ alone and the $K3$ surface containing $C$ is a result of some peculiarity of the canonical curve.  For instance \cite{BM}, if $[C]\in \K_g$ then the \emph{Wahl map} $$\psi_{K_C}:\wedge^2 H^0(C, K_C)\rightarrow H^0(C, K_C^{\otimes 3}),$$
is not surjective. It is natural to ask for similar intrinsic characterizations of the Prym-Nikulin locus $\cN_g\subset \cR_g$ in terms of Prym curves alone, without making reference to Nikulin surfaces. In this direction, we prove in Section 1 the following result:
\begin{theorem}
Set $g:=2i+6$. Then $K_{i, 2}(C, K_C\otimes \eta)\neq 0$ for any $[C, \eta]\in \cN_g$, that is, the Prym-canonical curve $C\stackrel{|K_C\otimes \eta|}\longrightarrow \PP^{g-2}$ of a Prym-Nikulin section fails to satisfy property $(N_i)$.
\end{theorem}
It is the content of the \emph{Prym-Green Conjecture} \cite{FL} that $K_{i, 2}(C, K_C\otimes \eta)=0$ for a general Prym curve $[C, \eta]\in \cR_{2i+6}$.
This suggests that curves on Nikulin surfaces can be recognized by extra syzygies of their Prym-canonical embedding.

\vskip 5pt Our initial motivation for considering Nikulin surfaces was to use them for the birational classification of moduli spaces of even theta-characteristics and we propose to turn our attention to the moduli space $\cS_g^+$ of even spin curves classifying pairs
$[C, \eta]$, where $[C]\in \cM_g$ is a smooth curve of genus $g$ and
$\eta\in \mbox{Pic}^{g-1}(C)$ is an even theta-characteristic. Let $\ss_g^+$ be the coarse moduli space associated to the Deligne-Mumford stack of
even stable spin curves of genus $g$, cf. \cite{C}. The projection $\pi: \cS_g^+ \rightarrow \cM_g$ extends to a
finite covering $\pi:\ss_g^+\rightarrow \mm_g$ branched along the boundary divisor $\Delta_0$ of $\mm_g$. It is shown in \cite{F1} that $\ss_g^+$ is a variety of general type as soon as $g\geq 9$.

\vskip 4pt
The existence of the dominant morphism $\chi_g:\P_g^{\mathfrak{N}}\rightarrow \cR_g$ when $g\leq 7$ and $g\neq 6$, leads to a straightforward uniruled parametrization of $\ss_g^+$, which we briefly describe. Let us start with a general even spin curve $[C, \eta]\in \cS_g^+$ and a non-trivial point of order two $e_C\in \mathrm{Pic}^0(C)$ in the Jacobian, such that $h^0(C, e_C\otimes \eta)\geq 1$. Since the curve $[C]\in \cM_g$ is general, it follows that $h^0(C, e_C\otimes \eta)=1$ and $Z:=\mathrm{supp}(e_C\otimes \eta)$ consists of $g-1$ distinct points. Applying Theorem \ref{uniruledpar}, if $g\neq 6$ there exists a Nikulin $K3$ surface $(S, e)$ containing $C$ such that $e_C=e\otimes \OO_C$. When $g=6$, there exists an Enriques surface
$(S, e)$ satisfying the same property, see \cite{V1}, and the construction described below goes through in that case as well. In the embedding
 $
 \phi_{|\OO_S(C)|}: S\rightarrow \PP^g,
 $
the span  $\langle Z\rangle\subset \PP^g $ is a codimension $2$ linear subspace and $h^0(S, \I_{Z/S}(1))=2$. Let
$$P:=\PP H^0\bigl(S, \I_{Z/S}(1)\bigr) \subset |\OO_S(C)|$$
be the corresponding pencil of curves on $S$. Each curve $D \in P$ is endowed with the odd theta-characteristic $\OO_D(Z)$. Twisting this line bundle  with $e\otimes \OO_D\in \mathrm{Pic}^0(D)$, we obtain an even theta-characteristic on $D$. This procedure induces a rational curve in moduli
$$
m: P \to \ss^+_g, \ \ \mbox{  }  P\ni D\mapsto [D, e\otimes \OO_D(Z)],
$$
which passes through the general point $[C, \eta]\in \ss_g^+$. This proves the following result:
\begin{theorem}\label{eventheta}
The moduli space $\ss^+_g$ is uniruled for $g \leq 7$.
\end{theorem}


It is known \cite{F1} that $\ss_g^+$ is of general type when $g\geq 9$. We complete the birational classification of $\ss_g^+$ and wish to highlight the following result, see Theorem \ref{spin8}:
\vskip 3pt
\begin{center}
\fbox{$\ss_8^+$ is a variety of Calabi-Yau type.}
\end{center}

We observe the curious fact that $\ss_8^-$ is unirational \cite{FV} whereas $\ss_8^+$ is not even uniruled.  In contrast to the case  of $\ss_g^{\mp}$, the birational classification of other important classes of moduli spaces is not complete. The Kodaira dimension of $\mm_g$ is unknown for $17\leq g\leq 21$,  see \cite{HM}, \cite{EH1}, the birational type of $\rr_g$ is not understood in the range $8\leq g\leq 13$, see \cite{FL}, whereas finding the Kodaira dimension of $\cA_6$ is a notorious open problem. Settling these outstanding cases is expected to require genuinely new ideas.
\vskip 4pt

The proof of Theorem \ref{spin8} relies on two main ideas: Following \cite{F1}, one finds an \emph{explicit} effective representative for the canonical divisor $K_{\ss_8^+}$ as a
$\mathbb Q$-combination of the divisor $\overline{\Theta}_{\mathrm{null}}\subset \ss_8^+$ of vanishing theta-nulls, the pull-back $\pi^*(\mm_{8, 7}^2)$ of the Brill-Noether divisor $\mm_{8, 7}^2$ on $\mm_8$ of curves with a $\mathfrak g^2_7$, and boundary divisor classes corresponding to spin curves whose underlying stable model is of compact type. This already implies the inequality $\kappa(\ss_8^+)\geq 0$.  Each irreducible component of this particular representative of $K_{\ss_8^+}$ is rigid (see Section 3), and the goal is to show that $K_{\ss_8^+}$ is rigid as well. To that end, we use the existence of a birational model $\mathfrak{M}_8$ of $\mm_8$
inspired by Mukai's work \cite{M2}. The space $\mathfrak{M}_8$ is realized as the following GIT quotient
$$\mathfrak{M}_8:=G(8, \wedge^2 V)^{\small{\mathrm{ss}}}\dblq SL(V),$$
where $V=\mathbb C^6$. We note that $\rho(\mathfrak{M}_8)=1$ and there exists a birational morphism
$$f:\mm_8\dashrightarrow \mathfrak{M}_8,$$ which contracts all the boundary divisors $\Delta_1, \ldots, \Delta_4$ as well as $\mm_{8, 7}^2$. Using the geometric description of $f$, we establish a geometric characterization of points inside $\thet$:
\begin{proposition} Let $C$ be a smooth curve of genus $8$ without a $\mathfrak g^2_7$. The following are equivalent:
\begin{itemize}
\item There exists a vanishing theta-null $L$ on $C$, that is, $[C, L]\in \thet$.
\item There exists a smooth $K3$ surface $S$ together with elliptic pencils $|F_1|$ and $|F_2|$ on $S$, such that $C\in |F_1+F_2|$ and $L=\OO_C(F_1)=\OO_C(F_2)$.
\end{itemize}
\end{proposition}
The existence of such a doubly elliptic $K3$ surface $S$ is equivalent to stating that there exists a smooth $K3$ extension $S\subset \PP^8$ of the canonical curve $C\subset \PP^7$, such that the rank three quadric $C\subset Q\subset \PP^7$ which induces the theta-null $L$, lifts to a \emph{rank 4} quadric $S\subset Q_S\subset \PP^8$. Having produced $S$, the pencils $|F_1|$ and $|F_2|$ define a product map
$$\phi:S\rightarrow \PP^1\times \PP^1,$$
such that each smooth member $D\in I:=|\phi^*\OO_{\PP^1\times \PP^1}(1, 1)|$ is a canonical curve contained in a rank $3$ quadric. A general pencil in $I$ passing through $C$ induces a rational curve $R\subset \ss_8^+$, and after intersection theoretic calculations on the stack $\ss_8^+$, we prove the following:

\begin{proposition}\label{pencil8}
The theta-null divisor $\thet\subset \ss_8^+$ is uniruled and swept by rational curves $R\subset \ss_8^+$ such that $R\cdot \thet<0$ and $R\cdot \pi^*(\mm_{8, 7}^2)=0$. Furthermore $R$ is disjoint from all boundary divisors $\pi^*(\Delta_i)$ for $i=1, \ldots, 4$.
\end{proposition}

Proposition \ref{pencil8} implies that $K_{\ss_8^+}$, expressed as a weighted sum of $\thet$, the pull-back $ \pi^*(\mm_{8, 7}^2)$ and boundary divisors $\pi^*(\Delta_i)$ for  $i=1, \ldots, 4$, is rigid as well. Equivalently, $\kappa(\ss_8^+)=0$. Note that since $K_{\ss_8^+}$ consists of $10$ uniruled base components which can be blown-down, the variety $\ss_8^+$ is not minimal and there exists a birational model $\mathcal{S}$ of $\ss_8^+$ which is a genuine Calabi-Yau variety in the sense that $K_{\mathcal{S}}=0$. Finding an explicit modular interpretation of this Calabi-Yau $21$-fold (or perhaps even its equations!) is a very interesting question.

\section{Prym-canonical curves on Nikulin surfaces}
 Let us start with a smooth $K3$ surface $Y$. A \emph{Nikulin involution} on $Y$ is an automorphism $\iota \in \mbox{Aut}(Y)$ of order $2$ which is symplectic, that is, $\iota^*(\omega)=\omega$, for all $\omega \in H^{2, 0}(Y)$. A Nikulin involution has  $8$ fixed points, see \cite{Ni} Lemma 3, and the quotient
$\bar{Y}:=Y/\langle \iota\rangle$ has $8$ ordinary double point singularities. Let $\sigma:\tilde{S}\rightarrow Y$ be the blow-up of the $8$ fixed points and denote by $E_1, \ldots, E_8\subset \tilde{S}$ the exceptional divisors and by $\tilde{\iota}\in \mbox{Aut}(\tilde{S})$ the automorphism induced by $\iota$. Then $S:=\tilde{S}/\langle \tilde{\iota}\rangle$ is a smooth $K3$ surface and if $f:\tilde{S}\rightarrow S$ is the projection, then $N_i:=f(E_i)$ are $(-2)$-curves on $S$. The branch divisor of $f$ is equal to $N:=\sum_{i=1}^8 N_i$.
We summarize the situation in the following diagram:
\begin{equation}\label{diagram}
\begin{CD}
{\tilde{S}} @>{\sigma}>> {Y} \\
@V{f}VV @V{}VV \\
{S} @>{}>> {\bar{Y}} \\
\end{CD}
\end{equation}
Sometimes we shall refer to the pair $(Y, \iota)$ as a Nikulin surface, while keeping the previous diagram in mind. We refer to \cite{Mo}, \cite{vGS} for a lattice-theoretic study on the action of the Nikulin involution on the cohomology $H^2(Y, \mathbb Z)=U^3\oplus E_8(-1)\oplus E_8(-1)$, where $U$ is the standard rank $2$ hyperbolic lattice and $E_8$ is the unique even, negative-definite unimodular lattice of rank $8$.
It follows from \cite{Mo} Theorem 5.7 that the orthogonal complement $E_8(-2)\cong \bigl(H^2(Y, \mathbb Z)^{\iota}\bigr)^\bot$ is contained in $\mathrm{Pic}(Y)$, hence $Y$ has Picard number at least $9$.
The class $\OO_S(N_1+\cdots+N_8)$ is divisible by $2$, and we denote by $e\in \mathrm{Pic}(S)$ the class such that $e^{\otimes 2}=\OO_S(N_1+\cdots+N_8)$.
\begin{definition} The \emph{Nikulin lattice} is an even lattice $\mathfrak{N}$ of rank $8$ generated by elements $\{\mathfrak{n}_i\}_{i=1}^8$ and $\mathfrak{e}:=\frac{1}{2}\sum_{i=1}^8 \mathfrak{n}_i$, with the bilinear form induced by $\mathfrak{n}_i^2=-2$ for $i=1, \ldots, 8$ and
$\mathfrak{n}_i\cdot \mathfrak{n}_j=0$ for $i\neq j$.
\end{definition}
Note that $\mathfrak N$ is the minimal primitive sublattice of $H^2(S, \mathbb Z)$ containing the classes $N_1, \ldots, N_8$ and $e$. For any Nikulin surface one has an embedding $\mathfrak{N}\subset \mathrm{Pic}(S)$. Assuming that $(Y, \iota)$ defines a general point in an irreducible component of the moduli space of  Nikulin involutions, both $Y$ and $S$ have Picard number 9 and there is a decomposition $\mbox{Pic}(S) = \mathbb Z\cdot [C] \oplus \mathfrak{N}$, where $C$ is an integral curve of genus
$g \geq 2$.  According to \cite{vGS} Proposition 2.2, only two cases are possible: either $C\cdot e = 0$ so that the previous decomposition is an orthogonal sum, or else, $C\cdot e \neq 0$, this second case being possible only when $g$ is odd. In this paper we consider only Nikulin surfaces of the first kind.

We fix an integer $g\geq 2$ and consider the lattice $\Lambda_g:=\mathbb Z\cdot \mathfrak c\oplus \mathfrak{N}$, where $\mathfrak c\cdot \mathfrak c=2g-2$.
\begin{definition}
A \emph{Nikulin surface of genus} $g$ is a $K3$ surface $S$ together with a primitive embedding of lattices $j:\Lambda_g\hookrightarrow \mathrm{Pic}(S)$ such that $C:=j(\mathfrak c)$ is a nef class.
\end{definition}

The coarse moduli space $\F_g^{\mathfrak{N}}$ of Nikulin surfaces of genus $g$ is the quotient of the $11$-dimensional domain
$$\mathcal{D}_{\Lambda_g}:=\{\omega \in \PP(\Lambda_g\otimes_{\mathbb Z}\mathbb C): \omega^2=0, \ \omega\cdot \bar{\omega}>0\}$$
by an arithmetic subgroup of ${\bf{O}}(\Lambda_g)$. Its existence follows e.g. from \cite{Do} Section 3.

\vskip 3pt
We now consider a Nikulin surface $f: \tilde{S} \rightarrow  S$, together with a smooth curve $C\subset S$ of genus $g$ such that $C\cdot N=0$.
If $\tilde C := f^{-1}(C)$, then
$f_C:=f_{|\tilde C}: \tilde C \rightarrow C
$ is an \'etale double covering.  By the Hodge index theorem, $\tilde C$ cannot split in two disjoint connected components,  hence $f_C$ is non-trivial and
$e_C:=\mathcal O_C(e)\in \mathrm{Pic}^0(C)$ is the non trivial 2-torsion element defining the covering $f_C$.  We set
$H \equiv C-e\in NS(S)$, hence $H^2=2g-6$ and $H\cdot C=2g-2$. For further reference we collect a few easy facts:
\begin{lemma}\label{easyfacts} Let $[S, e, \OO_S(C)]\in \F_g^{\mathfrak{N}}$ be a Nikulin surface such that $\mathrm{Pic}(S)=\Lambda_g$. The following statements hold:
\begin{enumerate}
\item $H^i(S, e)=0$ for all $i\geq 0$.
\item $\mathrm{Cliff}(C)=[\frac{g-1}{2}]$.
\item The line bundle $\OO_S(H)$ is ample for $g\geq 4$ and very ample for $g\geq 6$. In this range, it defines an embedding $\phi_H:S\rightarrow \PP^{g-2}$ such that the images $\phi_H(N_i)$ are lines for all $i=1, \ldots, 8$.
\item If $g\geq 7$, the ideal of the surface $\Phi_H(S)\subset \PP^{g-2}$ is cut out by quadrics.
\end{enumerate}
\end{lemma}
\begin{proof} Recalling that $e^{\otimes 2}=\OO_S(N_1+\cdots+N_8)$ and that the curves $\{N_i\}_{i=1}^8$ are pairwise disjoint, it follows that $H^0(S, e)=0$ and clearly $H^2(S, e)=0$. Since $e^2=-4$, by Riemann-Roch one finds that $H^1(S, e)=0$ as well.

In order to prove (ii) we assume that $\mathrm{Cliff}(C)<[\frac{g-1}{2}]$. From \cite{GL2} it follows that there exists a divisor
$D\in \mathrm{Pic}(S)$ such that $h^i(S, \OO_C(D))\geq 2$ for $i=0, 1$ and $C\cdot D\leq g-1$, such that $\OO_C(D)$ computes the Clifford index of $C$, that is,
$\mathrm{Cliff}(C)=\mathrm{Cliff}(\OO_C(D))$. But $C\cdot \ell\equiv 0 \ \mathrm{ mod } \ 2g-2$ for every class $\ell\in \mathrm{Pic}(S)$, hence no such divisor $D$ can exist.

Moving to (iii), the ampleness (respectively very ampleness) of $\OO_S(H)$ is proved in \cite{GS} Proposition 3.2 (respectively Lemma 3.1). From the exact sequence
$$
0 \longrightarrow \mathcal O_S(-H) \longrightarrow \mathcal O_S(e) \longrightarrow O_C(e) \longrightarrow 0,
$$
one finds that $h^1(S, \OO_S(H)) = 0$ and then $\mathrm{dim }|H| = g-2$. Furthermore $H\cdot N_i = 1$ for  $i = 1, \ldots, 8$ and the claim follows.

To prove (iv), following \cite{SD} Theorem 7.2, it suffices to show that there exists no irreducible curve $\Gamma \subset S$ with $\Gamma^2=0$ and
$H\cdot \Gamma=3$. Assume by contradiction that $\Gamma\equiv aC-b_1 N_1-\cdots -b_8 N_8$ is such a curve, where necessarily  $a, b_i\in \mathbb Z_{\leq 0}$. Then $\sum_{i=1}^8 b_i=2ag-2a-3$ and $\sum_{i=1}^8 b_i^2=a^2(g-1)$. Applying the Cauchy-Schwarz inequality $\bigl(\sum_{i=1}^8 b_i\bigr)^2\leq 8\bigl(\sum_{i=1}^8 b_i^2\bigr)$, we obtain an immediate contradiction.
\end{proof}

We consider the $\PP^g$-bundle $p_g:\P_g^{\mathfrak{N}}\rightarrow \F_g^{\mathfrak{N}}$, as well as the map
$$\chi_g:\P_g^{\mathfrak{N}}\rightarrow \cR_g, \ \ \mbox{  } \chi_g([S, e, C]):=[C, e_C:=e\otimes \OO_C]$$ defined in the introduction. We fix a  Nikulin surface $[S, e, \OO_S(C)]\in \P_g^{\mathfrak{N}}$. A Lefschetz pencil of curves $\{C_{\lambda}\}_{\lambda \in \PP^1}$ inside $|\OO_S(C)|$ induces a rational curve
$$\Xi_g:=\{[C_{\lambda}, \ e_{C_{\lambda}}:=e\otimes \OO_{C_{\lambda}}]:\lambda\in \PP^1\}\subset \rr_g.$$ In the range where $\chi_g$ is a dominant map, $\Xi_g$ is a rational curve passing through a general point of $\rr_g$, and it is of some interest to compute its numerical characters. If $\pi:\rr_g\rightarrow \mm_g$ denotes the projection map, we recall the formula \cite{FL} Example 1.4
\begin{equation}\label{pullbackrg}
\pi^*(\delta_0)=\delta_0^{'}+\delta_0^{''}+2\delta_{0}^{\mathrm{ram}},
\end{equation}
where $\delta_0^{'}:=[\Delta_0^{'}], \ \delta_0^{''}:=[\Delta_0^{''}]$ and $\delta_0^{\mathrm{ram}}:=[\Delta_0^{\mathrm{ram}}]$ are boundary divisor classes on $\rr_g$ whose meaning we recall. Let us fix a general point $[C_{xy}]\in \Delta_0$ induced by a $2$-pointed curve $[C, x, y]\in \cM_{g-1, 2}$ and the normalization map $\nu:C\rightarrow C_{xy}$, where $\nu(x)=\nu(y)$. A general point of $\Delta_0^{'}$ (respectively of $\Delta_0^{''}$) corresponds to a stable Prym curve $[C_{xy}, \eta]$, where $\eta\in \mathrm{Pic}^0(C_{xy})[2]$ and $\nu^*(\eta)\in \mathrm{Pic}^0(C)$ is non-trivial
(respectively, $\nu^*(\eta)=\OO_C$). A general point of $\Delta_{0}^{\mathrm{ram}}$ is of the form $[X, \eta]$, where $X:=C\cup_{\{x, y\}} \PP^1$ is a quasi-stable curve, whereas $\eta\in \mathrm{Pic}^0(X)$ is characterized by $\eta_{\PP^1}=\OO_{\PP^1}(1)$ and $\eta_C^{\otimes 2}=\OO_C(-x-y)$.
\begin{proposition}\label{intnumbers}
 If $\Xi_g\subset \rr_g$ is the curve induced by a pencil on a Nikulin surface, then
$$\Xi_g\cdot \lambda=g+1, \ \Xi_g\cdot \delta_0^{'}=6g+2, \ \  \Xi_g\cdot \delta_0^{''}=0\  \mbox{ and }\ \Xi_g\cdot \delta_0^{\mathrm{ram}}=8.$$
It follows that $\Xi_g\cdot K_{\rr_g}=g-15$.
\end{proposition}
\begin{proof} We use \cite{FP} Lemma 2.4 to find that \
$\Xi_g\cdot \lambda=\pi_*(\Xi_g)\cdot \lambda=g+1$ and $\Xi_g\cdot \pi^*(\delta_0)=\pi_*(\Xi_g)\cdot \delta_0=6g+18$, as well as \ $\Xi_g\cdot \pi^*(\delta_i)=0$ for $1\leq i\leq [g/2]$. For each $1\leq i\leq 8$, the sublinear system $\PP\ H^0(\OO_S(C-N_i))\subset \PP \ H^0(\OO_S(C))$ intersects $\Xi_g$ transversally in one point which corresponds to a curve $N_i+C_i\in |\OO_S(C)|$, where $N_i\cdot C_i=-N_i^2=2$ and $C_i\equiv C-N_i$. Furthermore $e\otimes \OO_{N_i}=\OO_{N_i}(1)$ and $e_{C_i}^{\otimes 2}=\OO_{C_i}(-N_i\cdot C_i)$. Each of these points lie in the intersection $\Xi_g\cap \Delta_0^{\mathrm{ram}}$. All remaining curves in $\Xi_g$ are irreducible, hence $\Xi_g\cdot \delta_0^{\mathrm{ram}}=8$. Since $\Xi_g\cdot \delta_0^{''}=0$, from (\ref{pullbackrg}) we find that $\Xi_g\cdot \delta_0^{'}=6g+2$.  Finally, according to \cite{FL} Theorem 1.5 the formula $K_{\rr_g}\equiv 13\lambda-2(\delta_0^{'}+\delta_0^{''})-3\delta_{0}^{\mathrm{ram}}-\cdots\in \mathrm{Pic}(\rr_g)$ holds, therefore putting everything together, \ $\Xi_g\cdot K_{\rr_g}=g-15$.
\end{proof}
The calculations in Proposition \ref{intnumbers} are applied now to show that syzygies of Prym-canonical curves on Nikulin surfaces are exceptional when compared to those of general Prym-canonical curves. To make this statement precise, let us recall the \emph{Prym-Green Conjecture}, see \cite{FL} Conjecture 0.7: If $g:=2i+6$ with $i\geq 0$, then the locus
$$\cU_{g, i}:=\{[C, \eta]\in \cR_{2i+6}: K_{i, 2}(C, K_C\otimes \eta)\neq 0\}$$
is a \emph{virtual divisor}, that is, the degeneracy locus of two vector bundles of the same rank defined over $\cR_{2i+6}$. The statement of the Prym-Green Conjecture is that this vector bundle morphism is generically non-degenerate:
\vskip 4pt

\noindent {\emph{Prym-Green Conjecture:}} $K_{i, 2}(C, K_C\otimes \eta)=0$ for a general Prym curve $[C, \eta]\in \cR_{2i+6}$.
\vskip 3pt

The conjecture is known to hold in bounded genus and has been used in \cite{FL} to show that $\rr_g$ is of general type when $g\geq 14$ is even.
\begin{theorem}\label{prymgreen}
For each $[S, e, C]\in \P_{2i+6}^{\mathfrak{N}}$ one has $K_{i, 2}(C, K_C\otimes e_C)\neq 0$. In particular, the Prym-Green Conjecture fails along the locus $\cN_{2i+6}$.
\end{theorem}
\begin{proof}If the non-vanishing $K_{i, 2}(C, K_C\otimes \eta)\neq 0$ holds for a general point $[C, \eta]\in \cR_{g}$, then there is nothing to prove, hence we may assume that $\cU_{g, i}$ is a genuine divisor on $\cR_{g}$. The class of its closure inside $\rr_{g}$ has been calculated \cite{FL} Theorem 0.6:
$$\overline{\cU}_{g, i}\equiv {2i+2\choose i} \ \Bigl(\frac{3(2i+7)}{i+3}\lambda-\frac{3}{2}\delta_0^{\mathrm{ram}}-\delta_0^{'}-\alpha\ \delta_0^{''}-\cdots\Bigr)\in \mathrm{Pic}(\rr_{2i+6}).$$
From Proposition \ref{intnumbers}, by direct calculation one finds that $\Xi_g\cdot \overline{\cU}_{g, i}=-{2i+3\choose i}<0$, thus \ $\Xi_g\subset \overline{\cU}_{g, i}$. By varying $\Xi_g$ inside $\rr_g$, we obtain that $\cN_g\subset \overline{\cU}_{g, i}$, which ends the proof.
\end{proof}
\begin{remark}\label{comment}
A geometric proof of Theorem \ref{prymgreen} using the Lefschetz hyperplane principle for Koszul cohomology is given in \cite{AF} Theorem 3.5. The indirect proof presented here is however shorter and illustrates how cohomology calculations on $\rr_g$ can be used to derive geometric consequences
for individual Prym curves.
\end{remark}

\begin{remark} One might ask whether similar applications to $\cR_g$ can be obtained using Enriques surfaces.  There is a major difference between Prym curves $[C, \eta]\in \cR_g$ lying on a Nikulin surface and those lying on an Enriques surface. For instance, if $C\subset S$ is a curve of genus $g$ lying on an Enriques surface $S$, then from \cite{CD} Corollary 2.7.1
$$\mbox{gon}(C)\leq 2\ \mathrm{inf}\bigl\{F\cdot C: F\in \mathrm{Pic}(S), \ F^2=0, \ F\not{\!\!\equiv} 0\bigr\}\leq\ 2\sqrt{2g-2}.$$
 In particular, for $g$  sufficiently high,  $C$ is far from being Brill-Noether general. On the other hand,  we have seen that for $[S, e, C]\in \P_g^{\mathfrak{N}}$ such that $\mathrm{Pic}(S)=\Lambda_g$, one has that $\mbox{gon}(C)=[\frac{g+3}{2}]$. For this reason, the \emph{Prym-Nikulin} locus $\cN_g:=\mbox{Im}(\chi_g)\subset \cR_g$ appears as a more promising and less constrained locus than the \emph{Prym-Enriques} locus in $\cR_g$, being transversal to stratifications of $\cR_g$ coming from Brill-Noether theory.
\end{remark}

\section{The Prym-Nikulin locus in $\cR_g$ for $g\leq 7$}
In this section we give constructive proofs of Theorems \ref{uniruledpar} and \ref{r7}. Comparing the dimensions $\mbox{dim}(\P_g^{\mathfrak{N}})=11+g$ and $\mbox{dim}(\cR_g)=3g-3$, one may inquire whether the morphism $\chi_g:\P_g^{\mathfrak{N}}\rightarrow \cR_g$ is dominant when $g\leq 7$. The similar question for ordinary $K3$ surfaces has been answered by Mukai \cite{M1}. Let $\F_g$ denote the $19$-dimensional moduli space of polarized $K3$ surfaces of genus $g$ and consider the associated $\PP^g$-bundle $$\P_g:=\bigl\{[S, C]: C\subset S \mbox{ is a smooth curve such that } [S, \OO_S(C)]\in \F_g\bigr\}.$$
The map
$q_g:\P_g\rightarrow \cM_g$ forgetting the $K3$ surface is dominant if and only if $g\leq 11$ and $g\neq 10$. The result for $g=10$ is  contrary to untutored expectation since the general fibre of $q_{10}$ is $3$-dimensional, hence $\mbox{dim}(\mbox{Im}(q_{10}))=\mbox{dim}(\P_{10})-3=26$. A strikingly similar picture emerges for Nikulin surfaces and Prym curves. The morphism $\chi_g:\P_g^{\mathfrak{N}}\rightarrow \cR_g$ is dominant when $g\leq 7$ and $g\neq 6$. For each genus we describe a geometric construction that furnishes a Nikulin surface in the fibre $\chi_g^{-1}\bigl([C, \eta]\bigr)$ over a general point $[C, \eta]\in \cR_g$.

\subsection{Nikulin surfaces of genus 7}

We start with a general element $[C, \eta]\in \cR_7$ and construct a Nikulin surface containing $C$.  One may assume that $\mbox{gon}(C)=5$ and that the line bundle $\eta$ does not lie in the difference variety $C_2-C_2\subset \mbox{Pic}^0(C)$, or equivalently, the linear series $L:=K_C\otimes \eta\in W^5_{12}(C)$ is very ample. It is a consequence of  \cite{GL} Theorem 2.1 that the Prym-canonical image
$
C \stackrel{|L|}\longrightarrow \PP^5
$
is quadratically normal, that is,  $h^0(\PP^5, \I_{C/\mathbf P^5}(2)) = 3$.
\begin{lemma}\label{aux7} For a general $[C, \eta]\in \cR_7$, the base locus of $|\I_{C/\PP^5}(2)|$ is a smooth K3 surface.
\end{lemma}
\begin{proof}  The property that the base locus of $|\I_{C/\PP^5}(2)|$ is smooth, is open in $\cR_7$ and it suffices to exhibit a single Prym-canonical curve $[C, \eta]\in \cR_7$ satisfying it.  Let us fix an element $(S, e, C)\in \P_7^{\mathfrak{N}}$ such that $\mathrm{Pic}(S)=\Lambda_7$ and set $H\equiv C-e$. Then according to Lemma \ref{easyfacts}, $\phi_H:S\rightarrow \PP^5$ is an embedding whose image $\phi_H(S)$ is ideal-theoretically cut out by quadrics. Moreover $\mbox{gon}(C)=5$, hence  $K_C\otimes e_C\in W^5_{12}(C)$ is quadratically normal. This implies that $H^0(S, \OO_S(2H-C))=H^1(S, \OO_S(2H-C))=0$, and then $H^0(\PP^5, \I_{S/\PP^5}(2))\cong H^0(\PP^5, \I_{C/\PP^5}(2))$, therefore the quadrics in $|\I_{C/\PP^5}(2)|$ cut out precisely the surface $S$.
\end{proof}
\begin{remark}\label{nikulin7}
This proof  shows that if $[S, e, C]\in \P_7^{\mathfrak{N}}$ is general then $\chi_7^{-1}\bigl([C, e_C]\bigr)=[S, e, C]$ and in particular the fibre $\chi_7^{-1}(\bigl[C, e_C]\bigr)$ is reduced.  Indeed, let $[S', e', C]\in \P_7^{\mathfrak{N}}$ be an arbitrary  Nikulin surface containing $C$. Set $H'\equiv C-e'\in NS(S')$. We may assume
that $\mathrm{Pic}(S')=\Lambda_7$, therefore the map $\phi_{H'}: S'\rightarrow \PP^5$ is an embedding whose image is cut out by quadrics. Since $\mbox{Cliff}(C)=3$, from Lemma \ref{easyfacts} we find that $K_C\otimes e_C$ is quadratically normal and then $S'$ is cut out by the quadrics contained in Prym-canonical embedding of $C\subset \PP^5$.
\end{remark}

Since both $\P_7^{\mathfrak{N}}$ and $\cR_7$ are irreducible varieties of dimension $18$, Remark \ref{nikulin7} shows that $\chi_7:\P_7^{\mathfrak{N}}\rightarrow \cR_7$ is a birational morphism and we now describe $\chi_7^{-1}$.
\begin{proposition}\label{gen7} For a general $[C, \eta]\in \cR_7$, the surface $S:=\mathrm{bs } \ |\I_{C/\PP^5}(2)|$ is a polarized Nikulin surface of genus $7$. \end{proposition}
\begin{proof} We show that $\mathrm{Pic}(S)\supset \mathbb Z\cdot C\oplus \mathfrak N$. Denote by $H\subset S$ the hyperplane class and let $N :\equiv 2(C-H)$, thus $N^2=-16$, $N\cdot H=8$ and $N\cdot C=0$. We aim to prove that $N$ is linearly equivalent to a sum of $8$ pairwise disjoint integral $(-2)$ curves on $S$. We consider the following exact sequence
$$
0 \longrightarrow \OO_S(N-C) \longrightarrow \OO_S(N) \longrightarrow \OO_C(N) \longrightarrow 0.
$$
Note that $\OO_C(N)$ is trivial because $e_C=\mathcal O_C(C-H)$ and that  $h^1(S, \OO_S(N-C)) = h^1(S, \OO_S(C-2H)) = 0$, because $C\subset \PP^5$ is quadratically normal. Passing to the long exact sequence, it follows that $h^0(S, \OO_S(N)) = 1$. Using Remark \ref{nikulin7} it follows that $N\equiv N_1+\cdots+N_8$, where $N_i\cdot N_j=-2\delta_{ij}$. Finally, to conclude that $[S, \mathbb Z\cdot C\oplus \mathfrak{N}]\in \F_7^{\mathfrak{N}}$ we must show that there is a primitive embedding $\mathbb Z\cdot C\oplus \mathfrak{N}\hookrightarrow \mbox{Pic}(S)$. We apply \cite{vGS} Proposition 2.7. Since $H^0(\tilde{S}, \OO_S(\tilde{C}))=H^0(S, \OO_S(C))\oplus H^0(S, \OO_S(C)\otimes e^{\vee})$ and sections in the second summand vanish on the exceptional divisor of the morphism $\sigma:\tilde{S}\rightarrow Y$, it follows that this is precisely the decomposition of $H^0(Y, \OO_Y(\tilde{C}))$ into $\iota_Y^*$-eigenspaces. Invoking \emph{loc. cit.}, we finish the proof.
\end{proof}

\subsection{The symmetric determinantal cubic hypersurface and Prym curves}
We provide a general set-up that allows us to reconstruct a Nikulin surface from a Prym curve of genus $g\leq 5$. Let us start with a curve $[C, \eta]\in \cR_g$ inducing an \'etale double cover $f:\tilde{C}\rightarrow C$ together with an involution $\iota:\tilde{C}\rightarrow \tilde{C}$ such that $f\circ \iota=f$. For each integer $r\geq -1$,  the
 \emph{Prym-Brill-Noether} locus is defined as the locus
$$V^r(C,\eta) := \{L\in \mathrm{Pic}^{2g-2}(\tilde{C}): \mathrm{Nm}_f(L) = K_C, \ h^0(L)\geq r+1 \hbox{ and } h^0(L)\equiv r+1 \ \mathrm{mod}\ 2\}.$$
Note that $V^{-1}(C, \eta)=\mathrm{Pr}(C, \eta)$.
For each line bundle $L\in V^r(C, \eta)$, the Petri map
$$\mu_0(L): H^0(\tilde{C}, L)\otimes H^0(\tilde{C}, K_{\tilde{C}}\otimes L^{\vee}) \rightarrow H^0(\tilde{C}, K_{\tilde{C}})$$ splits into an $\iota$-anti-invariant part
$$\mu_0^-(L): \Lambda^2 H^0(\tilde{C}, L) \rightarrow H^0(C, K_C\otimes \eta), \ \ \mbox{  } \ s\wedge t\mapsto s\cdot \iota^*(t)-t\cdot \iota^*(s),$$ and an $\iota$-invariant part respectively
$$\mu_0^+(L): \mathrm{Sym}^2H^0(\tilde{C}, L) \rightarrow H^0(C, K_C),\ \ \mbox{  }  s\otimes t+t\otimes s\mapsto s\cdot \iota^*(t)+t\cdot \iota^*(s).$$
For a general $[C, \eta]\in \cR_g$, the Prym-Petri map $\mu_0^-(L)$ is injective for every $L\in V^r(C, \eta)$ and $V^r(C, \eta)$ is equidimensional of dimension $g-1-{r+1\choose 2}$, see \cite{W}. We introduce the \emph{universal Prym-Brill-Noether variety}
$$\cR_g^r:=\Bigl\{\bigl([C, \eta], L\bigr): [C, \eta]\in \cR_g, \ \ L\in V^r(C, \eta)\Bigr\}.$$
When $g-1-{r+1\choose 2}\geq 0$, the variety $\cR_g^r$ is irreducible of dimension $4g-4-{r+1\choose 2}$.
We propose to focus on the case $r=2$ and $g\geq 4$ and choose a general triple $(f:\tilde{C}\rightarrow C, L)\in \cR_g^2$, such that $L$ is base point free and $h^0(\tilde{C}, L)=3$.
\vskip 2pt

Setting $\PP^2:=\PP\bigl(H^0(L)^{\vee}\bigr)$, we consider the quasi-\'etale double cover  $q:\PP^2\times \PP^2\rightarrow \PP^5$ obtained by projecting via the Segre embedding to the space of symmetric tensors. Note that $q$ is ramified along the diagonal $\Delta\subset \PP^2\times \PP^2$ and $V_4:=q(\Delta)\subset \PP^5$ is the Veronese surface. Moreover $\Sigma:=\mbox{Im}(q)$ is the determinantal symmetric cubic hypersurface isomorphic to the secant variety of $V_4$. We have the following  commutative diagram:
\begin{figure}[h]
$$\xymatrix@R=11pt{\tilde{C} \ar[rr]^-{(L, \iota^{*} L)} \ar[dd]_f && \PP^2\times \PP^2 \ar[dd]_q \ar@{^{}->}[rd] \\
&& & \PP^8=\PP\bigl(H^0 (L)^{\vee}\otimes H^0(L)^{\vee}\bigr) \ar@{-->}@<-1ex>[dl] \\
C\ar[rr]^-{\mu_0^+(L)} && *!<-22pt,0pt>{\PP^5=\PP(\mathrm{Sym}^2 H^0(L)^{\vee})}}
$$
\end{figure}

\noindent  Observe that the involution $\iota:\PP^8\rightarrow \PP^8$ given by $\iota[v\otimes w]:=[w\otimes v]$ where $v, w\in H^0(L)$, is compatible with $\iota:\tilde{C}\rightarrow \tilde{C}$. To summarize, giving a point $(\tilde{C}\rightarrow C, L)\in \cR_g^2$ is equivalent to specifying a symmetric determinantal cubic hypersurface $\Sigma \in H^0(\PP^{g-1}, \I_{C/\PP^{g-1}}(3))$ containing the canonical curve.

\subsection{A birational model of $\F_4^{\mathfrak{N}}$.}\ \
As a warm-up, we indicate how the set-up described above is a generalization of the construction that Catanese \cite{Ca} used to prove that $\cR_4$ is rational. For a general point $[C, \eta]\in \cR_4$ we find that $V^2(C, \eta)=\{L, \iota^*L\}$, that is, the pair $(L, \iota^*L)$ is uniquely determined. The map $\mu_0(L)$ has corank $2$ and $\PP^6_{\tilde{C}}:=\PP\bigl(\mathrm{Im }\ \mu_0(L)\bigr)\subset \PP^8$ has codimension $2$. The intersection $\tilde{T}:=(\PP^2\times \PP^2)\cap \PP^6_{\tilde{C}}$ is a del Pezzo surface of degree $6$, whereas $T:=\Sigma\cap \PP^3_+$ is a $4$-nodal Cayley cubic. Here we set $\PP^3_+:=\PP\bigl(H^0(K_C)^{\vee}\bigr)$. The double cover $q:\tilde{T}\rightarrow T$ is ramified at the singular points of $T$.
\vskip 4pt
To obtain a Nikulin surface containing $[C, \eta]$, we reverse this construction and start with a quartic rational normal curve $R\subset \PP^4$ and denote by $\bar{\mathcal{Y}}:=\mathrm{Sec}(R)\subset \PP^4$ its secant variety, which we view as a hyperplane section of $\Sigma\subset \PP^5$. Retaining the notation of diagram (\ref{diagram}), for a general quadric $Q\in |\OO_{\PP^4}(2)|$, the intersection $\bar{Y}:=\bar{\mathcal{Y}}\cap Q$ is a $K3$ surface with $8$ rational double points at $R\cap Q$. There exists a cover $q:Y\stackrel{2:1}\rightarrow \bar{Y}$ ramified at the singular points of $Y$, induced by restriction from the map $q:\PP^2\times \PP^2\rightarrow \Sigma$. Clearly $q:Y\rightarrow \bar{Y}$ is a Nikulin covering, and a hyperplane section in $|\OO_{\bar{Y}}(1)|$ induces a Prym curve $[C, \eta]\in \cR_4$ having general moduli. Moreover we have a birational isomorphism $$\F_4^{\mathfrak{N}}\stackrel{\cong}\dashrightarrow \PP\Bigl(H^0(\OO_{\PP^4}(2))\Bigr)^{\mathrm{ss}}\dblq SL_2,$$
where $PGL_2=\mbox{Aut}(R)\subset PGL_5$. An immediate consequence is that $\F_4^{\mathfrak{N}}$ is unirational.

\subsection{Nikulin surfaces of genus $3$.} We prove that $\chi_3:\P_3^{\mathfrak{N}}\rightarrow \cR_3$ is dominant and fix a complete intersection of $3$ quadrics $Y\subset \PP^5$ invariant with respect to an involution fixing a line $L\subset \PP^5$ and a $3$-dimensional linear subspace $\Lambda\subset \PP^5$. The projection $\pi_L:\PP^5\dashrightarrow \Lambda$ induces a quartic $\bar{Y}:=\pi_L(Y)$ with $8$ nodes, which is a Nikulin surface. We check that a general Prym curve $[C, \eta]\in \cR_3$ corresponding to an \'etale cover $f: \tilde C \to C$ embeds in such a surface.

Indeed, the canonical model $\tilde C \subset \PP^4$  is a complete intersection of $3$ quadrics. Fixing  projective coordinates on $\PP^4$, we can assume that the  involution $\iota: \tilde C \to \tilde C$ is induced by the projective involution $[x:y:u:v:t] \leftrightarrow [-x:-y:u:v:t]$. Note that the $\iota^*$anti-invariant quadratic forms are  vectors $q = ax + by$, where $a, b$ are linear forms in $u, v, t$. Since $\tilde C$ is complete intersection of 3 quadrics, no non-zero quadric $q = ax + by$ vanishes on $\tilde C$, for not, $\tilde C$ would intersect the plane $\lbrace x = y = 0 \rbrace$ and then $\iota$ would have  fixed points. Thus $\iota$ acts as the identity on the space $H^0(\PP^4, \I_{\tilde C/\PP^4}(2))$. Hence it follows $\tilde C = \lbrace a_1 + b_1 = a_2 + b_2 = a_3 + b_3 = 0 \rbrace$, where $a_i, b_i$ are quadratic forms  in $x,y$ and $u,v,t$. Passing to $\mathbf P^5$ by adding one coordinate $h$, we can choose quadratic forms $a_i + b_i + hl_i$, where $l_i$ is a general  linear form in $h, u, v, t$. Consider the surface  $Y \subset \mathbf P^5$ defined by the latter 3 equations. Then $[x:y:h:u:v:t] \leftrightarrow [-x:-y:h:u:v:t]$ induces a Nikulin involution on $Y$.  Let $\pi_L: Y \to \mathbf P^3$ be the projection of center $L = \lbrace h = u = v = t = 0 \rbrace$. Then $\overline Y := \pi_L(Y)$ is a quartic Nikulin surface and  $C = \pi_L(\tilde C)$ is a plane section of it.

\subsection{Nikulin surfaces of genus $5$}
To describe the morphism $\chi_5:\P_5^{\mathfrak{N}}\rightarrow \cR_5$ more geometrically, we use the set-up introduced in Subsection 2.2.
If $[C, \eta]\in \cR_5$ is general, then $\mbox{dim } V^2(C, \eta)=1$, the $\iota$-invariant Petri map $\mu_0^-(L)$ is injective, $\mu_0^+(L)$ surjective, thus $\mbox{dim}\bigl(\mbox{Coker } \mu_0(L)\bigr)=1$. We consider the hyperplane
$$\PP^7_{\tilde{C}}:=\PP\bigl(\mbox{Im}(\mu_0(L)\bigr)\subset \PP\bigl(H^0(L)^{\vee}\otimes H^0(L)^{\vee}\bigr)$$ and also set $\PP^4_+:=\PP\bigl(H^0(K_C)^{\vee}\bigr)\subset \PP^5$. Then we further denote
$$\tilde{T}:=(\PP^2\times \PP^2)\cap \PP^7_{\tilde{C}} \mbox{ and  }\  T:=\Sigma\cap \PP^4_+.$$ Note that $\tilde{T}$ is a degree $6$ threefold in $\PP^7_{\tilde{C}}$. Since the hyperplane $\PP^7_{\tilde{C}}$ is $\iota$-invariant, it follows $\tilde{T}$ is also endowed with the involution $\iota_{\tilde{T}}\in \mbox{Aut}(\tilde{T})$ such that $\mbox{Fix}(\iota_{\tilde{T}})=\Delta\cap \tilde{T}$ is a rational quartic curve in $\PP^4_+$.
Furthermore $T\subset \PP^4_+$ is the secant variety of $R$.
\begin{proposition}\label{gen5}
For a general point $[C, \eta, L]\in \cR_5^2$ the following statements hold:
\begin{enumerate}
\item The threefold $\tilde{T}\subset \PP^2\times \PP^2$ is smooth, while $T\subset \PP^4_+$  is singular precisely along $R$.
\item $h^0(\tilde{T}, \I_{\tilde{C}/\tilde{T}}(2))=3$. Moreover $H^i(\tilde{T}, \I_{\tilde{C}/\tilde{T}}(2))=0$ for $i=1, 2$.
\item Every quadratic section in the linear system $|\I_{\tilde{C}/\tilde{T}}(2)|$ is $\iota$-invariant, that is, $$H^0(\tilde{T}, \I_{\tilde{C}/\tilde{T}}(2))=q^* H^0(T, \I_{C/T}(2)).$$
\item A general quadratic section $Y\in |\I_{\tilde{C}/\tilde{T}}(2)|$ is a smooth $K3$ surface endowed with an involution $\iota_{Y}$ with fixed points precisely at the  $8$ points in the intersection $R\cap Y$.
\end{enumerate}
\end{proposition}
\begin{proof}
We take cohomology in the following exact sequence
$$0\longrightarrow \I_{\tilde{C}/\PP^2\times \PP^2}(2)\longrightarrow \OO_{\PP^2\times \PP^2}(2)\longrightarrow K_{\tilde{C}}^{\otimes 2}\longrightarrow 0,$$ to note that
$h^0(\I_{\tilde{C}/\tilde{T}}(2))=3 (\Leftrightarrow H^1(\I_{\tilde{C}/\PP^2\times \PP^2}(2))=0$), if and only if the composed map
$$\mathrm{Sym}^2 H^0(\tilde{C}, L)\otimes \mathrm{Sym}^2 H^0(\tilde{C}, \iota^*L)\rightarrow H^0(\tilde{C}, L^{\otimes 2})\otimes H^0(\tilde{C}, \iota^*(L^{\otimes 2}))\rightarrow H^0(\tilde{C}, K_{\tilde{C}}^{\otimes 2})$$
is surjective. This is an open condition and a triple $(\tilde{C}\stackrel{f}\rightarrow C, L)\in \cR_g^2$ satisfying it, and for which moreover $\tilde{T}\subset \PP^2\times \PP^2$ is smooth, has been constructed in \cite{V2} Section 4. Finally, from the exact sequence
 $$0\longrightarrow \I_{T/\PP^4_+}(2)\longrightarrow \I_{C/\PP^4_+}(2)\longrightarrow \I_{C/T}(2)\rightarrow 0,$$
 we compute that $h^0(T, \I_{C, T}(2))=3$, therefore $q^*: H^0(T, \I_{C/T}(2))\rightarrow H^0(\tilde{T}, \I_{\tilde{C}/\tilde{T}}(2))$ is an isomorphism, based on dimension count.
 Part (iv) is a consequence of (i)-(iii). Assume that $\bar{Y}=T\cap Q$, where $Q\in H^0(\I_{C/\PP^4_+}(2))$. Then $Y=\tilde{T}\cap q^*(Q)$ and  $\bar{Y}$ is the quotient of $Y$ by the involution $\iota_Y$ obtained by restriction from $\iota\in \mbox{Aut}(\PP^2\times \PP^2)$. It follows that the covering $q:Y\rightarrow \bar{Y}$ is a Nikulin surface such that $C\subset \bar{Y}\subset \PP^4_+$. To conclude, we must check that for a general choice of $Y\in |\I_{\tilde{C}/\tilde{T}}(2)|$, the point $[Y, \iota_Y]$ gives rise to an element of $\F_5^{\mathfrak N}$, that is, using the notation of diagram (\ref{diagram}), that $\mbox{Pic}(S)=\Lambda_5$.  Proposition 2.7 from \cite{vGS} picks out two possibilities for $\mbox{Pic}(Y)$
 (or equivalently for $\mbox{Pic}(S)$), and we must check that $\mathbb Z\cdot \OO_Y(\tilde{C})\oplus E_8(-2)$ has index $2$ in $\mbox{Pic}(Y)$, see also \cite{GS} Corollary 2.2. \footnote{We are grateful to the referee for raising this point that we have initially overlooked.}

 This is achieved by finding the decomposition of $H^0(\OO_Y(\tilde{C}))$ into $\iota_Y^*$-eigenspaces. In the course of the proof of \cite{V2} Proposition 5.2 an example of a smooth quadratic section $Y\in |\I_{\tilde{C}/\tilde{T}}(2)|$ is constructed such that
 $$H^0(Y, \OO_Y(\tilde{C}))^{+}=q^*H^0(\bar{Y}, \OO_{\bar{Y}}(C)).$$
 In particular the $(+1)$-eigenspace of $H^0(Y,\OO_Y(\tilde{C}))$ is $6$-dimensional and invoking once more \cite{vGS} Proposition 2.7, we conclude that $[Y, \iota_Y]\in \F_5^{\mathfrak{N}}$.
\end{proof}

We close this subsection with an amusing result on a geometric divisor on $\cR_5$. For a Prym curve $[C, \eta]\in \cR_5$ and $L:=K_C\otimes \eta\in W^3_8(C)$, we observe that the vector spaces entering the multiplication map
$\nu_3(L):\mbox{Sym}^3 H^0(C, L)\rightarrow H^0(C, L^{\otimes 3})$
have the same dimension. The condition that $\nu_3(L)$ be not an isomorphism is divisorial in $\cR_5$. We have not been able to find a direct proof of the following equality of cycles on $\cR_5$, even though one inclusion is straightforward:
\begin{theorem}\label{cubics}
Let $[C, \eta]\in \cR_5$ be a Prym curve such that the Prym-canonical line bundle $K_C\otimes \eta$ is very ample. Then $\phi_{K_C\otimes\eta}: C\rightarrow \PP^3$ lies on a cubic surface if and only if $C$ is trigonal.
\end{theorem}
\begin{proof} Let $\mathfrak{D}_1$ be the locus of Prym curves whose Prym-canonical model lies on a cubic
$$\mathfrak{D}_1:=\{[C, \eta]\in \rr_5: \nu_3(\omega_C\otimes \eta): \mathrm{Sym}^3 H^0(C, \omega_C\otimes \eta)\stackrel{\ncong}\longrightarrow  H^0(C, \omega_C^{\otimes 3}\otimes \eta^{\otimes 3})\},$$
and $\mathfrak{D}_2$ the closure inside $\rr_5$ of the divisor $\{[C, \eta]\in \cR_5: \eta\in C_2-C_2\}$ of smooth Prym curves for which  $L:=K_C\otimes \eta\in W^3_8(C)$ is not very ample. Obviously, $\mathfrak{D}_1-\mathfrak{D}_2\geq 0$, for if $L$ is not very ample, then the multiplication map $\nu_3(L):\mbox{Sym}^3 H^0(C, L)\rightarrow H^0(C, L^{\otimes 3})$ cannot be an isomorphism. The class of $\mathfrak{D}_2$ can be read off \cite{FL} Theorem 5.2:
$$\mathfrak{D}_2\equiv 14\lambda-2(\delta_0^{'}+\delta_0^{''})-\frac{5}{2}\delta_0^{\mathrm{ram}}-\cdots \ \in \mathrm{Pic}(\rr_5).$$
For $i\geq 1$, let $\mathbb E_i$ be the vector bundle over $\rr_5$ with fibre $\mathbb E_i[C, \eta]=H^0(C, \omega_C^{\otimes i}\otimes \eta^{\otimes i})$ for every $[C, \eta]\in \rr_g$. One has the following formulas from \cite{FL} Proposition 1.7:
$$c_1(\mathbb E_i)={i\choose 2}(12\lambda-\delta_0^{'}-\delta_0^{''}-2\delta_0^{\mathrm{ram}})+\lambda-\frac{i^2}{4}\delta_0^{\mathrm{ram}}\in \mathrm{Pic}(\rr_g).$$ As a consequence, $\mathfrak{D}_1 \equiv c_1(\mathbb E_3)-c_1(\mathrm{Sym}^3 \mathbb E_1)\equiv 37\lambda-3(\delta_0+\delta_0^{''})-\frac{33}{4}\delta_0^{\mathrm{ram}}-\cdots\in \mathrm{Pic}(\rr_5),$
 therefore \ $\mathfrak{D}_1-\mathfrak{D}_2\equiv 8\lambda-(\delta_0^{'}+\delta_0^{''})-2\delta_0^{\mathrm{ram}}-\cdots=\pi^*(8\lambda-\delta_0-\cdots)\geq 0,$ where the terms left out are combinations of boundary divisors $\pi^*(\delta_i)$ with $i\geq 1$, corresponding to reducible curves. The only effective divisors $D\equiv a\lambda-b_0\delta_0-b_1\delta_{1}-b_2\delta_2$ on $\mm_5$ such that $\frac{a}{b_0}\leq 8$ and satisfying $\Delta_i\nsubseteq \mathrm{supp}(D)$ for $i=1, 2$,  are multiples of the trigonal locus $\mm_{5, 3}^1$ (the proof is identical to that of Proposition \ref{septics}). This proves that if $[C, \eta]\in \mathfrak{D}_1-\mathfrak{D_2}$, with $C$ being a smooth curve, then necessarily $[C]\in \cM_{5, 3}^1$, which finishes the proof.
\end{proof}

\section{A singular quadratic complex and a birational model for $\F_6^{\mathfrak{N}}$}

Let us set $V:=\mathbb C^{n+1}$ and denote by $\GG:=G(2, V)\subset \PP(\wedge^2 V)$ the Grassmannian of lines in $\PP(V)$. We fix once and for all a smooth quadric  $Q \subset \PP(V)$. The projective tangent bundle
$\mathbf P_Q := \mathbf P(T_Q)$ can be realized as the incidence correspondence
$$
\PP_Q=\bigl\{ (x, \ell) \in Q \times \GG: x\in \ell\subset \PP(T_x Q) \bigr\}.
$$
For each point $x \in Q$, the fibre $\mathbf P_{Q}(x)$ is the space of lines tangent to $Q$ at $x$. We introduce the projections
$p:\PP_Q \rightarrow \GG$ and $q:\PP_Q\rightarrow Q$, then set
$$W_Q:=p(\PP_Q)=\{\ell\in \GG: \ell \ \mbox{ is tangent to the quadric } Q\}.$$ Note that $W_Q$ contains the Hilbert scheme of lines in $Q$, which we denote by
$V_Q\subset W_Q$. It is well-known that $V_Q$ is smooth, irreducible and  $\mbox{dim}(V_Q)=2n-5$. The restriction $p_{| p^{-1}(W_Q-V_Q)}$ is an isomorphism and
$E_Q := p^{-1}(V_Q)\subset \PP_Q$ is the exceptional divisor of $p$.

\begin{proposition}\label{quadraticcomplex}
The variety $W_Q$ is a quadratic complex of lines in $\GG$. Its singular locus is equal to $V_Q$ and each point of $V_Q$ is an ordinary double point of $W_Q$.
\end{proposition}
\begin{proof}
Let $Q:V\rightarrow \mathbb C$ be the quadratic form whose zero locus is the quadric hypersurface also denoted by $Q\subset \PP(V)$, and $\tilde{Q}:V\times V\rightarrow \mathbb C$ the associated bilinear map. We define the bilinear map $\nu_2(\tilde{Q}):\wedge^2 V\times \wedge^2 V\rightarrow \mathbb C$ by the formula
$$\nu_2(\tilde{Q})(u\wedge v, s\wedge t):=\tilde{Q}(u, s)\tilde{Q}(v, t)-\tilde{Q}(v, s)\tilde{Q}(u, t)$$
for $u, v, s, t\in V$, and denote by $\nu_2(Q):\wedge^2 V\rightarrow \mathbb C$ the induced quadratic form.

For fixed points $x=[u]\in Q$ and $y=[v]\in \PP(V)$, we observe that the line $\ell=\langle x, y \rangle $ is tangent to $Q$ if and only if $\tilde{Q}(u, v)=0\Leftrightarrow \nu_2(Q)(u\wedge v)=0$. Therefore $W_Q=\GG\cap \nu_2(Q)$ is a quadratic line complex in $\GG$, being the vanishing locus of $\nu_2(Q)$.

Keeping the same notation, a point $\ell=[u\wedge v]\in W_Q$ is a singular point, if and only if the linear form $\nu_2(\tilde{Q})(u\wedge v,\ -)$ vanishes along
$\PP(T_{\ell}\  \GG)$. Since $\PP(T_{\ell}\  \GG)$ is spanned by the Schubert cycle $\{m\in \GG: m\cap \ell\neq \emptyset\}$, any tangent vector in $T_{\ell}(\GG)$ has a representative of the form $u\wedge a-v\wedge b$, where $a, b\in V$. We obtain that $[u\wedge v]\in \mathrm{Sing}(W_Q)$ if and only if $Q(v, v)=0$, that is, $\ell=[u\wedge v]\in V_Q$. Since $W_Q$ is a quadratic complex, each point $\ell\in V_Q$ has multiplicity $2$.
\end{proof}
The map $p:\PP_Q\rightarrow W_Q$ appears as a desingularization of the quadratic complex $W_Q$. We shall compute the class of the exceptional divisor $E_Q$ of $\PP_Q$.
Let $H:=p^*(\OO_{\GG}(1))$ be the class of the family of tangent lines to $Q$ intersecting a fixed $(n-2)$-plane in $\PP(V)$ and $B:=q^* (\OO_Q(1))\in \mathrm{Pic}(\PP_Q)$. Furthermore,  we consider the class  $h\in NS_1(\PP_Q)$ of the pencil of tangent lines to $Q$ with center a given point $x \in Q$. It is clear that
$$
h\cdot H = 1 \ \  \mbox{ and } \ \ h\cdot B = 0.
$$
If $\ell\in V_Q$ is a fixed line, let $s\in NS_1(\PP_Q)$ be the class of the family $\{(x,\ell): x \in \ell \}$.  Then
$$
s\cdot H = 0 \ \mbox{ and } \ s\cdot B = 1.
$$
\begin{lemma} The linear equivalence $E_Q \equiv 2H - 2B$ in $\mathrm{Pic}(\PP_Q)$ holds. In particular, the class $E_Q$ is divisible by 2 and it is the branch divisor of a double cover
$$
f: \tilde {\mathbf P}_Q \to \mathbf P_Q.
$$
\end{lemma}
\begin{proof} To compute the class of $E_Q$ it suffices to
compute $h\cdot E_Q$ and $s\cdot E_Q$. First we note that $h\cdot E_Q = 2$. Indeed a pencil of tangent lines to $Q$ through a fixed point $x \in Q$ has two elements which are in $Q$.
Finally, recalling that $V_Q=\mathrm{Sing}(W_Q)$ consists of ordinary double points,  we obtain that $s\cdot E_Q = -2$, since $p^{-1}(\ell)$ is a conic inside $\PP\bigl(N_{V_Q/\GG}(\ell)\bigr)$.
\end{proof}
\vskip 4pt
\subsection{A birational model for $\F_6^{\mathfrak{N}}$.} Let us now specialize to the case $n=4$, that is,
$$Q\subset \PP^4, \ \ \GG=G(2, 5)\subset \PP^9 \mbox{ and } \ \mbox{dim}(W_Q)=5.$$ The class of $V_Q$ equals $4\sigma_{2, 1}\in H^6(\GG, \mathbb Z)$ see \cite{HP} p. 366, therefore $\mbox{deg}(V_Q)=4\sigma_{2, 1}\cdot \sigma_1^3=8$. This can also be seen by recalling that $V_Q$ is isomorphic to the Veronese $3$-fold $\nu_2(\PP^3)\subset \PP^9$.

The double covering $f: \tilde {\mathbf P}_Q \rightarrow \PP_Q$ constructed above has a transparent projective interpretation.
For  $(x, \ell)\in \PP_Q$, we denote by $\Pi_{\ell} \in G(3, V)$ the \emph{polar space} of $\ell$ defined as the base locus of the pencil of polar hyperplanes $\{z\in \PP(V): \tilde{Q}(y, z)=0\}_{y\in \ell}$. Clearly $x\in \Pi_{\ell}\subset \PP(T_x Q)$ and $Q\cap \Pi_{\ell}$ is a conic of rank at most $2$ in $\Pi_{\ell}$. When $\ell\in W_Q-V_Q$, the quadric has rank exactly $2$ which corresponds to a pair of lines $\ell_1+\ell_2$ with $\ell_1, \ell_2\in V_Q$. The double cover is induced by the map from the parameter space of the lines themselves.

In the next statement we shall keep in mind the notation of diagram (\ref{diagram}):
\begin{proposition}\label{genul6}
A general codimension $3$ linear section $\bar{Y}:=\Lambda\cap W_Q$ of the quadratic complex $W_Q$ where $\Lambda\in G(7, \wedge ^2 V)$, is a $8$-nodal $K3$ surface with desingularization $$p:S:=p^{-1}(\bar{Y})\rightarrow \bar{Y}.$$ The triple $[S, \OO_S(H-B), \OO_S(H)]\in \F_6^{\mathfrak{N}}$ is a Nikulin surface of genus $6$ and the induced double cover is the restriction $f:\tilde{S}:=f^{-1}(S)\rightarrow S$.
\end{proposition}
\begin{proof} We fix a general $6$-plane $\Lambda\in G(7, \wedge ^2 V)$. Since $K_{W_Q}=\OO_{W_Q}(-3H)$, by adjunction we obtain that $\bar{Y}:=\Lambda\cap W_Q$ is a $K3$ surface. From Bertini's theorem, $\bar{Y}$ has ordinary double points at the $8$ points of intersection $\Lambda\cap V_Q$.
General hyperplane sections of $C\in |\OO_{\bar{Y}}(H)|$, viewed as codimension $4$ linear sections of $W_Q$, are canonical curves of genus $6$, endowed with a line bundle of order $2$ given by $\OO_C(H-B)$. The remaining statements are immediate.
\end{proof}
It turns out that the general Nikulin surface of genus $6$ arises in this way:
\begin{theorem}\label{param6}
Let $V:=\mathbb C^5$ and $Q\subset \PP(V)$ be a smooth quadric. One has a dominant map
$$\varphi: G(7, \wedge^2 V)^{\mathrm{ss}}\dblq \mathrm{Aut}(Q)\dashrightarrow \F_6^{\mathfrak{N}},$$ given by
$\varphi(\Lambda):=\bigl[S:=p^{-1}(\Lambda\cap W_Q), \ \OO_S(H-B), \ \OO_S(H)\bigr]$.
\end{theorem}
\begin{proof}
Via the embedding $\mbox{Aut}(Q)\subset PGL(V)\hookrightarrow PGL(\wedge^2 V)$, we observe that every automorphism of $Q$ induces an automorphism of $\PP(\wedge^2 V)$ that fixes both $W_Q$ and $V_Q$. Since (i) the moduli space $\F_6^{\mathfrak{N}}$ is irreducible and (ii) polarized Nikulin surfaces have finite automorphism groups, it suffices to observe that $\mbox{dim } G(7, \wedge^2 V)\dblq \mbox{Aut}(Q)=21-10=11$ and $\mbox{dim}(\F_6^{\mathfrak{N}})=11$ as well.
\end{proof}
\begin{corollary} The Prym-Nikulin locus $\cN_6\subset \cR_6$ is an irreducible unirational divisor, which is set-theoretically equal to the ramification locus of the Prym map $\mathrm{Pr}:\cR_6\rightarrow \cA_5$
$$\cU_{6, 0}=\{[C, \eta]\in \cR_6: K_{0, 2}(C, K_C\otimes \eta)\neq 0\}.$$
Furthermore, the exists a dominant rational map $G(6, \wedge^2 V)^{\mathrm{ss}}\dblq \mathrm{Aut}(Q)\dashrightarrow \cN_6.$
\end{corollary}
\begin{proof} Just observe that $\langle C \rangle = \mathbf P^5$ and that this has codimension 4 in $\PP(\wedge^2 V)$,  hence there is a $\PP^3$ of
Nikulin sections of $W_Q$ containing $C$.
\end{proof}

The divisor $\kk_{10}\subset \mm_{10}$ of sections of $K3$ surfaces is known to be an extremal point of the effective cone $\mbox{Eff}(\mm_{10})$. An analogous result holds for the closure of $\cN_6$:
\begin{proposition}\label{extremal6}
The Prym-Nikulin divisor $\overline{\mathcal{N}}_6$ is extremal in the effective cone $\mathrm{Eff}(\rr_6)$:
\end{proposition}
\begin{proof} It follows from \cite{FL} Theorem 0.6 that $\overline{\mathcal{N}}_6\equiv 7\lambda-\frac{3}{2}\delta_0^{\mathrm{ram}}-(\delta_0^{'}+\delta_0^{''})-\cdots\in \mathrm{Pic}(\rr_6)$. The divisor $\overline{\mathcal{KN}}_6$
is filled-up by the rational curves $\Xi_6\subset \rr_6$ constructed in the course of proving Theorem \ref{intnumbers}. We compute that
\ $\Xi_6\cdot \overline{\mathcal{N}}_6=-1$, which completes the proof.
\end{proof}

\section{Spin curves and the divisor $\thet$}

We turn our attention to the moduli space of spin curves and begin by setting notation and terminology. If $\bf{M}$ is a Deligne-Mumford stack, we denote by $\cM$ its associated coarse moduli space. A $\mathbb Q$-Weil divisor $D$ on a normal $\mathbb Q$-factorial projective variety $X$ is said to be \emph{movable} if $\mbox{codim}\bigl(\bigcap_{m} \mbox{Bs}|mD|, X\bigr)\geq 2$, where the intersection is taken over all $m$ which are sufficiently large and divisible. We say that $D$ is \emph{rigid} if $|mD|=\{mD\}$, for all $m\geq 1$ such that $mD$ is an integral Cartier divisor. The \emph{Kodaira-Iitaka dimension} of a divisor $D$ on $X$ is denoted by $\kappa(X, D)$.

If $D=m_1D_1+\cdots +m_sD_s$ is an effective $\mathbb Q$-divisor on $X$, with irreducible components $D_i\subset X$  and $m_i> 0$ for $i=1, \ldots, s$, a (trivial) way of showing
that $\kappa(X, D)=0$ is by exhibiting for each $1\leq i\leq s$, an irreducible curve $\Gamma_i\subset X$ passing through a general point of $D_i$, such that $\Gamma_i\cdot D_i<0$ and $\Gamma_i\cdot D_j=0$ for $i\neq j$.

\vskip 4pt

We recall basic facts about the moduli space $\ss_g^+$ and refer to \cite{C}, \cite{F1} for details.
\begin{definition}
An \emph{even spin curve} of genus
$g$ consists of a triple $(X, \eta, \beta)$, where $X$ is a genus
$g$ quasi-stable curve, $\eta\in \mathrm{Pic}^{g-1}(X)$ is a line
bundle of degree $g-1$ such that $\eta_{E}=\OO_E(1)$ for every
rational component $E\subset X$ with $|E\cap (\overline{X-E})|=2$ (such a component is called \emph{exceptional}), \ $h^0(X, \eta)\equiv 0
\mbox{ mod } 2$, and   $\beta:\eta^{\otimes
2}\rightarrow \omega_X$ is a morphism of sheaves which is generically
non-zero along each non-exceptional component of $X$.
\end{definition}
Even spin curves of genus $g$ form a smooth Deligne-Mumford stack $\pi:\overline{\bf{S}}_g^+\rightarrow \overline{\bf{M}}_g$. At the level of coarse moduli schemes, the morphism $\pi:\ss_g^+\rightarrow \mm_g$ is the stabilization map
$\pi([X, \eta, \beta]):=[\mathrm{st}(X)]$, which associates to a quasi-stable curve its stable model.

We explain the boundary structure of $\ss_g^+$: If $[X, \eta, \beta]\in \pi^{-1}([C\cup_y D])$,
where $[C, y]\in \cM_{i, 1}, [D, y]\in \cM_{g-i, 1}$ and $1\leq i\leq [g/2]$, then necessarily
$X=C\cup_{y_1} E\cup_{y_2} D$, where $E$ is an exceptional
component such that $C\cap E=\{y_1\}$ and $D\cap E=\{y_2\}$.
Moreover $\eta=\bigl(\eta_C, \eta_D, \eta_E=\OO_E(1)\bigr)\in
\mbox{Pic}^{g-1}(X)$, where $\eta_C^{\otimes 2}=K_C, \eta_D^{\otimes
2}=K_D$. The condition $h^0(X, \eta)\equiv 0 \mbox{ mod } 2$,
implies that the theta-characteristics $\eta_C$ and $\eta_D$ have
the same parity. We denote by $A_i\subset \ss_g^+$ the closure of
the locus corresponding to pairs $$([C,y, \eta_C], [D, y, \eta_D])\in
\cS_{i, 1}^+\times \cS_{g-i, 1}^+$$ and by $B_i\subset \ss_g^+$ the
closure of the locus corresponding to pairs $$([C, y, \eta_C], [D, y,
\eta_D])\in \cS_{i, 1}^-\times \cS_{g-i, 1}^{-}.$$

We set
$\alpha_i:=[A_i]\in \mathrm{Pic}(\ss_g^+), \beta_i:=[B_i]\in
\mathrm{Pic}(\ss_g^+)$, and then one has the relation
\begin{equation}\pi^*(\delta_i)=\alpha_i+\beta_i.
\end{equation}

We recall the description of the ramification divisor of the covering $\pi:\ss_g^+\rightarrow \mm_g$. For a point $[X, \eta, \beta]\in \ss_g^+$ corresponding to a stable model $\mbox{st}(X)=C_{yq}:=C/y\sim q$,
with $[C, y, q]\in \cM_{g-1, 2}$, there are two possibilities
depending on whether $X$ possesses an exceptional component or not.
If $X=C_{yq}$ (i.e. $X$ has no exceptional component) and $\eta_C:=\nu^*(\eta)$ where $\nu:C\rightarrow X$
denotes the normalization map, then $\eta_C^{\otimes 2}=K_C(y+q)$.
For each choice of $\eta_C\in \mathrm{Pic}^{g-1}(C)$ as above, there
is precisely one choice of gluing the fibres $\eta_C(y)$ and
$\eta_C(q)$ such that $h^0(X, \eta) \equiv 0 \mbox{ mod } 2$. We denote by $A_0$ the
closure in $\ss_g^+$ of the locus of spin curves $[C_{yq}, \eta_C\in
\sqrt{K_C(y+q)}]$ as above.

If $X=C\cup_{\{y, q\}} E$, where $E$ is an exceptional component,
then $\eta_C:=\eta\otimes \OO_C$ is a theta-characteristic on $C$.
Since $H^0(X, \omega)\cong H^0(C, \omega_C)$, it follows that $[C,
\eta_C]\in \cS_{g-1}^{+}$.  We denote by
$B_0\subset \ss_g^+$ the closure of the locus of spin curves
$$\bigl[C\cup_{\{y, q\}} E, \ E\cong \PP^1, \ \eta_C\in \sqrt{K_C}, \ \eta_E=\OO_E(1)\bigr]\in \cS_g^+.$$ If
$\alpha_0:=[A_0], \beta_0:=[B_0]\in
\mbox{Pic}(\ss_g^+)$,  we have the relation, see \cite{C}:
\begin{equation}\label{del0}
\pi^*(\delta_0)=\alpha_0+2\beta_0.
\end{equation} In particular, $B_0$ is the ramification divisor of $\pi$.
An important  effective divisor on $\ss_g^{+}$ is the locus of vanishing theta-nulls
$$\Theta_{\mathrm{null}}:=\{[C, \eta]\in \cS_g^{+}: H^0(C, \eta)\neq
0\}.$$ The class of its compactification inside $\ss_g^+$
is given by the formula, cf.
\cite{F1}:
\begin{equation}\label{thetanull}
\overline{\Theta}_{\mathrm{null}}\equiv
\frac{1}{4}\lambda-\frac{1}{16}\alpha_0-\frac{1}{2}\sum_{i=1}^{[g/2]}
 \beta_i\in \   \mathrm{Pic}(\ss_g^{+}).
 \end{equation}
 It is also useful to recall the formula for the canonical class of $\ss_g^+$:
$$K_{\ss_g^{+}}\equiv\pi^*(K_{\mm_g})+\beta_0 \equiv
13\lambda-2\alpha_0-3\beta_0-2\sum_{i=1}^{[g/2]}
(\alpha_i+\beta_i)-(\alpha_1+\beta_1).$$

An argument involving spin curves on certain singular canonical surfaces in $\PP^6$, implies that for $g\leq 9$, the divisor $\thet$ is uniruled and a rigid point in the cone of effective
divisors $\mathrm{Eff}(\ss_g^+)$:

\begin{theorem}\label{extremalthetanull}
For $g\leq 9$ the divisor $\thet\subset \ss_g^+$ is uniruled and
rigid. Precisely, through a general point of $\thet$ there passes a rational curve $\Gamma \subset \ss_g^+$ such that $\Gamma \cdot \thet<0$. In particular, if  $D$ is an effective divisor on $\ss_g^+$ with $D\equiv n\thet$ for some $n\geq 1$, then $D=n\thet$.
\end{theorem}

\begin{proof} A general point $[C, \eta_C]\in \Theta_{\mathrm{null}}$ corresponds to a canonical
curve $C\stackrel{|K_C|}\hookrightarrow \PP^{g-1}$ lying on a rank $3$ quadric $Q\subset \PP^{g-1}$ such that
$C\cap \mathrm{Sing}(Q)=\emptyset$. The pencil $\eta_C$ is recovered from the ruling of $Q$. We construct the pencil $\Gamma\subset \ss_g^+$ by representing $C$ as a section of a nodal canonical surface $S\subset Q$ and noting that $\mbox{dim}\ |\OO_S(C)|=1$. The construction of $S$ depends on the genus and we describe the various cases separately.
\vskip 3pt

\noindent {\bf{(i)}} $7\leq g\leq 9$.
We choose $V\in G\bigl(7, H^0(C, K_C)\bigr)$  such that if $\pi_V:\PP^{g-1}\dashrightarrow \PP(V^{\vee})$ denotes the projection, then
 $\tilde{Q}:=\pi_V(Q)$ is a quadric of rank $3$. Let $C':=\pi_V(C)\subset \PP(V^{\vee})$ be the projection of the canonical curve $C$.
 By counting dimensions we find that
$$\mbox{dim}\Bigl\{I_{C'/\PP(V^{\vee})}(2):=\mbox{Ker}\bigl\{\mathrm{Sym}^2(V)\rightarrow H^0(C, K_C^{\otimes 2})\bigl\}\Bigr\}\geq 31-3g\geq 4,$$
that is, the embedded curve $C'\subset \PP^6$ lies on at least $4$ independent quadrics, namely the rank $3$ quadric $\tilde{Q}$ and $Q_1, Q_2, Q_3\in |I_{C'/\PP(V^{\vee})}(2)|$.
By choosing $V$ sufficiently general we make sure that  $S:=\tilde{Q}\cap Q_1\cap Q_2 \cap Q_3$ is a canonical surface in $\PP(V^{\vee})$ with $8$ nodes corresponding to the intersection $\bigcap_{i=1}^3 Q_i\cap \mathrm{Sing}(\tilde{Q})$ (This transversality statement can also be checked with Macaulay by representing $C$ as a section of the corresponding Mukai variety). From the exact sequence on $S$,
$$0\longrightarrow \OO_S\longrightarrow \OO_S(C)\longrightarrow \OO_C(C)\longrightarrow 0,$$
coupled with the adjunction formula $\OO_C(C)=K_C\otimes K_{S |C}^{\vee}=\OO_C$, as well as the fact $H^1(S, \OO_S)=0$, it follows that
$\mbox{dim }|C|= 1$, that is, $C\subset S$ moves in its linear system. In particular, $\thet$ is a  uniruled divisor for $g\leq 9$.

We determine the numerical parameters of the family $\Gamma \subset \ss_g^+$ induced by varying $C\subset S$. Since $C^2=0$, the pencil $|C|$ is base point free and gives rise to a fibration $f:\tilde{S}\rightarrow \PP^1$, where $\tilde{S}:=\mathrm{Bl}_8(S)$ is the blow-up of the nodes of $S$. This in turn induces a moduli map $m:\PP^1\rightarrow \ss_g^+$ and $\Gamma=:m(\PP^1).$  We have the formulas $$\Gamma \cdot \lambda=m^*(\lambda)=\chi(S, \OO_S)+g-1=8+g-1=g+7,\ \hfill$$ and
$$\Gamma\cdot \alpha_0+2\Gamma \cdot \beta_0=m^*(\pi^*(\delta_0))=m^*(\alpha_0)+2m^*(\beta_0)=c_2(\tilde{S})+4(g-1).$$
Noether's formula gives that
$c_2(\tilde{S})=12\chi(\tilde{S}, \OO_{\tilde{S}})-K_{\tilde{S}}^2=12\chi(S, \OO_S)-K_S^2=80$, hence $m^*(\alpha_0)+2m^*(\beta_0)=4g+76.$ The singular fibres  corresponding to spin curves lying in $B_0$ are those in the fibres over the blown-up nodes and all contribute with multiplicity $1$, that is, $\Gamma \cdot \beta_0=8$ and then $\Gamma \cdot \alpha_0=4g+60$.
It follows that $\Gamma\cdot \thet=-2<0$ (independent of $g$!), which finishes the proof.

\vskip 3pt
\noindent {\bf{(ii)}} $g=5$. In the case $C\subset Q\subset \PP^4$ and we choose a general quartic  $X\in H^0(\PP^4, \I_{C/\PP^4}(4))$ and set $S:=Q\cap X$. Then $S$ is a canonical surface with nodes at the $4$ points  $X\cap \mathrm{Sing}(Q)$. As in the previous case $\mbox{dim}\ |C|=1$, and the numerical characters of the induced family $\Gamma\subset \ss_5^+$ can be readily computed:
$$\Gamma\cdot \lambda=g+5=10,\ \Gamma\cdot \beta_0=|\mathrm{Sing}(S)|=4, \ \mbox{and } \ \Gamma\cdot \alpha_0=4g+52,$$
where the last equality is a consequence of Noether's formula $\Gamma\cdot (\alpha_0+2\beta_0)=12\chi(S, \OO_S)-K_S^2+4(g-1)=4g+60$.
By direct calculation, we obtain once more that $\Gamma\cdot \thet=-2$. The case $g=6$ is similar, except that the canonical surface $S$ is a $(2, 2, 3)$ complete intersection in $\PP^5$, where one of the quadrics is the rank $3$ quadric $Q$.

\vskip 3pt
\noindent {\bf{(iii)}} $g=4$. In this last case we proceed slightly differently and denote by $S=\mathbb F_2$ the blow-up of the vertex of a cone $Q\subset \PP^3$ over a conic in $\PP^3$ and write $\mathrm{Pic}(S)=\mathbb Z\cdot F+\mathbb Z\cdot C_0$, where $F^2=0$, $C_0^2=-2$ and $C_0\cdot F=1$. We choose a Lefschetz pencil of genus $4$ curves in the linear system $|3(C_0+2F)|$. By blowing-up the $18=9(C_0+2F)^2$ base points, we obtain a fibration $f:\tilde{S}:=\mathrm{Bl}_{18}(S)\rightarrow \PP^1$ which induces a family of spin curves $m:\PP^1\rightarrow \ss_4^+$ given by $m(t):=[f^{-1}(t), \OO_{f^{-1}(t)}(F)]$. We have the formulas
$$m^*(\lambda)=\chi(\tilde{S}, \OO_{\tilde{S}})+g-1=4, \   \mbox{ }\mbox{ and }$$
$$m^*(\pi^*(\delta_0))=m^*(\alpha_0)+2m^*(\beta_0)=c_2(\tilde{S})+4(g-1)=34.$$
 The singular fibres lying in $B_0$ correspond to curves in the Lefschetz pencil on $Q$ passing through the vertex of the cone, that is, when $f^{-1}(t_0)$ splits as $C_0+D$, where $D\subset \tilde{S}$ is the residual curve. Since $C_0\cdot D=2$ and $\OO_{C_0}(F)=\OO_{C_0}(1)$, it follows that $m(t_0)\in B_0$. One finds that $m^*(\beta_0)=1$, hence $m^*(\alpha_0)=32$ and
$m^*(\thet)=-1$. Since $\Gamma:=m(\PP^1)$ fills-up the divisor $\thet$, we obtain that $[\thet]\in \mathrm{Eff}(\ss_4^+)$ is rigid.
\end{proof}

\section{Spin curves of genus $8$}

The moduli space $\cM_8$ carries one Brill-Noether divisor, the locus of plane septics $$\cM_{8, 7}^2:=\{[C]\in \cM_8: G_7^2(C)\neq \emptyset\}.$$ The locus $\mm_{8, 7}^2$ is irreducible and for a known constant $c_{8, 7}^2\in \mathbb Z_{>0}$, one has, cf. \cite{EH1},
$$\mathfrak{bn}_8:=\frac{1}{c_{8, 7}^2} \mm_{8, 7}^2\equiv 22\lambda-3\delta_0-14\delta_1-24\delta_2-30\delta_3-32\delta_4\in \mbox{Pic}(\mm_8).$$
In particular, $s(\mm_{8, 7}^2)=6+12/(g+1)$ and this is the minimal slope of an effective divisor on $\mm_8$. The following fact is probably well-known:

\begin{proposition}\label{septics}
Through a general point of $\mm_{8, 7}^2$ there passes a rational curve $R\subset \mm_8$ such that $R\cdot \mm_{8, 7}^2<0$. In particular, the class $[\mm_{8, 7}^2]\in \mathrm{Eff}(\mm_8)$ is rigid.
\end{proposition}
\begin{proof}
One takes a Lefschetz pencil of nodal plane septic curves with $7$ assigned nodes in general position (and $21$ unassigned base points). After blowing up the $21$ unassigned base points as well as the $7$ nodes, we obtain a fibration $f:S:=\mathrm{Bl}_{28}(\PP^2)\rightarrow \PP^1$, and the corresponding moduli map $m:\PP^1\rightarrow \mm_8$ is a covering curve for the irreducible divisor $\mm_{8, 7}^2$. The numerical invariants of this pencil are
$$m^*(\lambda)=\chi(S, \OO_S)+g-1=8\ \mbox{ and } m^*(\delta_0)=c_2(S)+4(g-1)=59,$$
while  $m^*(\delta_i)=0$ for $i=1, \ldots, 4$. We find  $m^*([\mm_{8, 7}^2])=c^2_{8, 7}(8\cdot 22-3\cdot 59)=-c^2_{8, 7}<0$.
\end{proof}

Using (\ref{thetanull}) we find the following explicit representative for the canonical class $K_{\ss_8^+}$:
\begin{equation}\label{can}
K_{\ss_8^+}\equiv \frac{1}{2} \pi^*(\mathfrak{bn}_8)+8\thet+\sum_{i=1}^4 (a_i\ \alpha_i+b_i\ \beta_i),
\end{equation}
where $a_i, b_i>0$ for $i=1, \ldots, 4$. The multiples of each irreducible component appearing in (\ref{can})  are rigid divisors on $\ss_8^+$, but in principle, their sum could still be a movable class. Assuming for a moment Proposition \ref{pencil8}, we explain how this implies Theorem \ref{spin8}:
\vskip 5pt

\noindent \emph{Proof of Theorem \ref{spin8}.} The covering curve $R\subset \thet$ constructed in Proposition \ref{pencil8}, satisfies $R\cdot \thet<0$ as well as $R\cdot \pi^*(\mm_{8, 7}^2)=0$ and $R\cdot \alpha_i=R\cdot \beta_i=0$ for $i=1, \ldots, 4$. It follows from (\ref{can}) that for each $n\geq 1$, one has an equality of linear series on $\ss_8^+$
$$|nK_{\ss_8^+}|=8n\thet+|n(K_{\ss_8^+}-8\thet)|.$$

\noindent Furthermore, from (\ref{can}) one finds  constants $a_i'>0$ for $i=1, \ldots, 4$, such that if $$D\equiv 22\lambda-3\delta_0-\sum_{i=1}^4 a_i'\ \delta_i\in \mbox{Pic}(\mm_8),$$ then the difference $\frac{1}{2}\pi^*(D)-(K_{\ss_8^+}-8\thet)$ is still effective on $\ss_8^+$. We can thus write
$$0\leq \kappa(\ss_8^+)=\kappa\bigl(\ss_8^+, K_{\ss_8^+}-8\thet\bigr)\leq \kappa\bigl(\ss_8^+, \frac{1}{2}\pi^*(D)\bigr)=\kappa\bigl(\ss_8^+, \pi^*(D)\bigr).$$
We claim that $\kappa\bigl(\ss_8^+, \pi^*(D)\bigr)=0$.  Indeed, in the course of the proof of Proposition \ref{septics} we have constructed a covering family $B\subset \mm_8$ for the divisor $\mm_{8, 7}^2$ such that $B\cdot \mm_{8, 7}^2<0$ and $B\cdot \delta_i=0$ for $i=1, \ldots, 4$. We lift $B$ to a family $R\subset \ss_8^+$ of spin curves by taking
$$\tilde{B}:=B\times _{\mm_8} \ss_8^+=\{[C_t, \ \eta_{C_{t}}]\in \ss_8^{+}:
[C_{t}]\in B, \eta_{C_{t}}\in \overline{\mbox{Pic}}^{7}(C_{t}), t
\in \PP^1 \}\subset \ss_8^+.$$ One notes that $\tilde{B}$ is disjoint from the boundary divisors $A_i, B_i\subset \ss_8^+$ for $i=1, \ldots, 4$, hence  $\tilde{B}\cdot \pi^*(D)=2^{g-1}(2^g+1) (B\cdot \mm_{8, 7}^2)_{\mm_8}<0$. Thus we write that
$$\kappa\bigl(\ss_8^+, \pi^*(D)\bigr)=\kappa\bigl(\ss_8^+, \pi^*(D-(22\lambda-3\delta_0)\bigr)=\kappa \bigl(\ss_8^+, \sum_{i=1}^4 a_i'(\alpha_i+\beta_i)\bigr)=0.$$
$ \hfill \Box$

\section{A family of spin curves  $R \subset \ss^+_8$ with $R \cdot \pi^*(\mm_{8,7}^2) = 0$ and $R \cdot \thet <0$}
The aim of this section is to prove Proposition \ref{pencil8}, which is the key ingredient in the proof of Theorem \ref{spin8}. We begin by reviewing  facts about the geometry of $\mm_8$, in particular the construction of general curves of genus 8 as complete intersections in a rational homogeneous variety, see \cite{M2}.

We fix $V:=\mathbb C^6$ and denote by $\GG:=G(2, V)\subset \PP(\wedge^2 V)$ the Grassmannian of lines. Noting that smooth codimension $7$ linear sections of $\GG$ are canonical curves of genus $8$, one is led to consider the \emph{Mukai model} of the moduli space of curves of genus $8$
$$\mathfrak{M}_8:=G(8, \wedge^2 V)^{\mathrm{st}}\dblq SL(V).$$
There is a birational map $f:\mm_8\dashrightarrow \mathfrak{M}_8$, whose inverse is given by $f^{-1}(H):=\GG\cap H$, for a general $H\in G(8, \wedge^2 V)$. The map $f$ is constructed as follows:  Starting with a curve $[C]\in \cM_8-\cM_{8, 7}^2$, one notes that $C$ has a finite number of pencils $\mathfrak g^1_5$. We choose $A\in W^1_5(C)$ and set $L:=K_C\otimes A^{\vee}\in W^3_9(C)$. There exists a unique rank $2$ vector bundle $E\in SU_C(2, K_C)$ (independent of $A$!), sitting in an extension
 $$0\longrightarrow A\longrightarrow E\longrightarrow L\longrightarrow 0,$$
 such that $h^0(E)=h^0(A)+h^0(L)=6$. Since $E$ is globally generated, we  define the map
 $$\phi_E: C\rightarrow G\bigl(2, H^0(E)^{\vee}\bigr), \ \mbox{    } \ \mbox{ } \phi_E(p):=E(p)^{\vee} \ \bigl(\hookrightarrow H^0(E)^{\vee}\bigr),$$ and let $\wp: G(2, H^0(E)^{\vee})\rightarrow \PP(\wedge^2 H^0(E)^{\vee})$ be the Pl\"ucker embedding.  The determinant map $u:\wedge ^2 H^0(E)\rightarrow H^0(K_C)$ is surjective and we can view $H^0(K_C)^{\vee}\in G(8, \wedge^2 H^0(E)^{\vee})$, see \cite{M2} Theorem C. We set
 $$f([C]):=H^0(K_C)^{\vee} \ \mathrm{ mod }\  SL(H^0(E)^{\vee}) \in \mathfrak{M}_8,$$
 that is, we assign to $C$  its linear span $\langle C\rangle$ under the Pl\"ucker map
 $\wp\circ \phi_E:C\rightarrow \PP\bigl(\wedge^2 H^0(E)^{\vee}\bigr)$.

  Even though this is not strictly needed for our proof, it follows from \cite{M2} that the exceptional divisors of $f$ are the Brill-Noether locus $\mm_{8, 7}^2$ and the boundary divisors $\Delta_1, \ldots, \Delta_4$. The map $f^{-1}$ does not contract any divisors.
 \vskip 5pt

Inside the moduli space $\F_8$ of polarized $K3$ surfaces $[S, h]$ of degree $h^2=14$, we consider the following \emph{Noether-Lefschetz} divisor
$$\mathfrak{NL}:=\{[S, \OO_S(C_1+C_2)]\in \F_8: \mathrm{Pic}(S)\supset \mathbb Z\cdot C_1\oplus \mathbb Z\cdot C_2,\  \  C_1^2=C_2^2=0, \ C_1\cdot C_2=7\},$$
of doubly-elliptic $K3$ surfaces. For a general element $[S, \OO_S(C)]\in \mathfrak{NL}$, the embedded surface $\phi_{\OO_S(C)}:S\hookrightarrow \PP^8$ lies on a rank $4$ quadric whose rulings induce the elliptic pencils $|C_1|$ and $|C_2|$ on $S$.

Let $\cU\rightarrow \mathfrak{NL}$ be the space classifying pairs $\bigl([S, \OO_S(C_1+C_2)], C\subset S\bigr)$, where $$C\in |H^0(S, \OO_S(C_1))\otimes H^0(S, \OO_S(C_2))|\subset |H^0(S, \OO_S(C_1+C_2))|.$$ An element of $\cU$ corresponds to a hyperplane section $C\subset S\subset \PP^8$ of a doubly-elliptic $K3$ surface, such that the intersection of $\langle C\rangle $ with the rank $4$ quadric induced by the elliptic pencils, has rank at most $3$. There exists a rational map $$q:\cU\dashrightarrow \thet, \ \ \ \mbox{ } \ q\bigl([S, \OO_S(C_1+C_2)], C\bigr):=[C, \OO_C(C_1)=\OO_C(C_2)].$$ Since $\cU$ is birational to a $\PP^3$-bundle over an open subvariety of $\mathfrak{NL}$, we obtain that $\cU$ is irreducible and $\mbox{dim}(\cU)=21 \bigl(=3+\mbox{dim}(\mathfrak{NL})\bigr)$. We shall show that the morphism $q$ is dominant (see Corollary \ref{explicitfibre}) and begin with some preparations.
\vskip 5pt

We fix a general point  $[C, \eta]\in \thet\subset \ss_8^+$, with $\eta$ a vanishing theta-null. Then
$$
C \subset Q \subset \PP^7:=\PP\bigl(H^0(C, K_C)^{\vee}\bigr),
$$
where $Q\in H^0(\PP^7, \mathcal{I}_{C/\PP^7}(2))$ is the rank $3$ quadric such that the ruling of $Q$ cuts out on $C$ precisely  $\eta$.
As explained, there exists a linear embedding $\PP^7 \subset \PP^{14}:=\PP\bigl(\wedge^2 H^0(E)^{\vee}\bigr)$  such that
$\PP^7 \cap \GG = C$. The restriction map yields an isomorphism between spaces of quadrics, cf. \cite{M2},
$$\mathrm{res}_C: H^0(\GG, \mathcal{I}_{\GG/\PP^{14}}(2))\stackrel{\cong}\longrightarrow H^0(\PP^7, \mathcal{I}_{C/\PP^7}(2)).$$
In particular there is a unique quadric $\GG\subset \tilde{Q} \subset \PP^{14}$ such that $\tilde{Q} \cap \PP^7=Q$.

There are three possibilities for the rank of any quadric $\tilde{Q}\in H^0(\PP^{14}, \mathcal{I}_{\GG/\PP^{14}}(2))$:
\  (a) $\mathrm{rk}(\tilde{Q})=15$, \ (b) $\mathrm{rk}(\tilde{Q})=6$ and then $\tilde{Q}$ is a \emph{Pl\"ucker quadric}, or \ (c) $\mbox{rk}(\tilde{Q})=10$, in which case $\tilde{Q}$ is a sum of two Pl\"ucker quadrics, see \cite{M2} Proposition 1.4.
\vskip4pt
\begin{proposition} For a general $[C, \eta]\in \thet$, the quadric $\tilde{Q}$ is smooth, that is, $\mathrm{rk}(\tilde{Q})=15$.
\end{proposition}
\begin{proof} We may assume that $\mbox{dim }G^1_5(C)=0$ (in particular $C$ has no $\mathfrak g^1_4$'s), and  $G^2_7(C)=\emptyset$. The space $\PP(\mathrm{Ker}(u))\subset \PP\bigl(\wedge^2 H^0(E)\bigr)$ is identified with the space of hyperplanes $H\in (\PP^{14})^{\vee}$ containing the canonical space $\PP^7$.
\vskip 3pt

\noindent {\emph{Claim:}} If $\mbox{rk}({\tilde{Q}})<15$, there exists a pencil of $8$-dimensional planes $\PP^7\subset \Xi \subset \PP^{14}$, such that $S:=\GG\cap \Xi$ \ is a  $K3$ surface containing $C$ as a hyperplane section, and
$$
\mathrm{rk}\bigl\{Q_{\Xi}:=\tilde{Q}\cap \Xi\in H^0(\Xi, \mathcal{I}_{S/\Xi}(2))\bigr\}=3.
$$

The conclusion of the claim contradicts the assumption that $[C, \eta]\in \thet$ is general. Indeed, we pick such an $8$-plane $\Xi$ and corresponding $K3$ surface $S$. Since $\mbox{Sing}(Q)\cap C=\emptyset$, where $Q_{\Xi}\cap \PP^7=Q$, it follows that $S\cap \mbox{Sing}(Q_{\Xi})$ is finite. The ruling of $Q_{\Xi}$ cuts out an elliptic pencil $|E|$ on $S$. Furthermore, $S$ has nodes at the points $S\cap \mbox{Sing}(Q_{\Xi})$. For numerical reasons, $|\mathrm{Sing}(S)|=7$, and then on the surface $\tilde{S}$ obtained from $S$ by resolving the $7$ nodes, one has the linear equivalence $$C\equiv 2E+\Gamma_1+\cdots +\Gamma_7,$$ where $\Gamma_i^2=-2, \ \Gamma_i\cdot E=1$ for $i=1, \ldots, 7$ and $\Gamma_i\cdot \Gamma_j=0$ for $i\neq j$. In particular $\mbox{rk}(\mbox{Pic}(\tilde{S}))\geq 8$. A standard parameter count, see e.g. \cite{Do}, shows that
 $$\mathrm{dim}\bigl\{(S, C): C\in |\OO_S(2E+\Gamma_1+\cdots+ \Gamma_7)|\bigr\}\leq 19-7+\mbox{dim}|\OO_{\tilde{S}}(C)|=20.$$
 Since $\mbox{dim}(\thet)=20$ and a general curve $[C]\in \thet$ lies on infinitely many such $K3$ surfaces $S$, one obtains a contradiction.

 \vskip 4pt
We are left with proving the claim made in the course of the proof. The key point is to describe the intersection $\PP(\mbox{Ker}(u))\cap \tilde{Q}^{\vee}$, where we recall that the linear span $\langle \tilde{Q}^{\vee} \rangle$ classifies  hyperplanes $H\in (\PP^{14})^{\vee}$ such that
$\mbox{rk}(\tilde{Q}\cap H)\leq \mbox{rk}(\tilde{Q})-1$. Note also that $\mbox{dim } \langle \tilde{Q}\rangle=\mbox{rk}(\tilde{Q})-2$.

If $\mbox{rk}(\tilde{Q})=6$, then $\tilde{Q}^{\vee}$ is contained in the dual Grassmannian $\GG^{\vee}:=G(2, H^0(E))$, cf. \cite{M2} Proposition 1.8. Points in the intersection $\PP(\mbox{Ker}(u))\cap \GG^{\vee}$ correspond to decomposable tensors $s_1\wedge s_2$, with $s_1, s_2\in H^0(C, E)$,  such that $u(s_1\wedge s_2)=0$. The image of the morphism $\OO_C^{\oplus 2}\stackrel{(s_1, s_2)}\longrightarrow E$ is thus a subbundle $\mathfrak g^1_5$ of $E$ and there is a bijection
$$ \PP(\mathrm{Ker}(u))\cap \GG\bigl(2, H^0(E)\bigr)\cong  W^1_5(C).$$ It follows, there are at most finitely many tangent hyperplanes to $\tilde{Q}$ containing the space $\PP^7=\langle C\rangle$, and consequently, $\mbox{dim}\bigl(\PP(\mbox{Ker}(u))\cap \langle \tilde{Q}^{\vee}\rangle\bigr)\leq 1$. Then there exists a codimension $2$ linear space $W^{12}\subset \PP^{14}$ such that $\mbox{rk}(\tilde{Q}\cap W)=3$, which proves the claim (and much more), in the case $\mbox{rk}(\tilde{Q})=6$.

When $\mbox{rk}(\tilde{Q})=10$, using the explicit description of the dual quadric $\tilde{Q}^{\vee}$ provided in \cite{M2} Proposition 1.8, one finds that  $\mbox{dim}\bigl(\PP(\mbox{Ker}(u))\cap \langle \tilde{Q}^{\vee}\rangle\bigr)\leq 4$. Thus there exists a codimension $5$ linear section
$W^9\subset \PP^{14}$ such that $\mbox{rk}(\tilde{Q}\cap W)=3$, which implies the claim when $\mbox{rk}(\tilde{Q})=10$ as well.

\end{proof}

We consider an $8$-dimensional linear extension
$\PP^7\subset \Lambda^8 \subset \PP^{14}$ of the canonical space $\PP^7=\langle C\rangle$, such that
$S_{\Lambda} := \Lambda \cap \GG$
is a smooth K3 surface. The restriction map
$$\mathrm{res}_{C/S_{\Lambda}}: H^0(\Lambda, \mathcal I_{S_{\Lambda}/\Lambda}(2)) \to H^0(\PP^7, \mathcal I_{C / \PP^7}(2))$$
is an isomorphism, see \cite{SD}. Thus there exists a \emph{unique}  quadric
$S_{\Lambda}\subset Q_{\Lambda} \subset \Lambda$
with $Q_{\Lambda} \cap \PP^7 = Q$.
Since $\mbox{rk}(Q)= 3$, it follows that $3 \leq \mbox{rk}(Q_{\Lambda}) \leq 5$ and it is easy to see that for a general
$\Lambda$, the corresponding quadric $Q_{\Lambda}\subset \Lambda$ is of rank $5$. We show however, that one can find $K3$-extensions of the canonical curve $C$, which lie on quadrics of rank $4$:

\begin{proposition}\label{exprk4}  For a general $[C, \eta]\in \thet$, there exists a pencil of $8$-dimensional extensions $$\PP(H^0(C, K_C)^{\vee})\subset \Lambda\subset \PP^{14}$$ such that $\mathrm{rk}(Q_{\Lambda})=4$. It follows that there exists a smooth $K3$ surface $S_{\Lambda}\subset \Lambda$ containing $C$ as a transversal hyperplane section, such that $\mathrm{rk}(Q_{\Lambda})=4$.
\end{proposition}
\begin{proof} We pass from projective to vector spaces and view the rank $15$ quadric $$\tilde{Q}: \wedge^2 H^0(C, E)^{\vee}\stackrel{\sim}\longrightarrow \wedge^2 H^0(C, E)$$ as an isomorphism, which by restriction to $H^0(C, K_C)^{\vee}\subset \wedge^2 H^0(C, E)^{\vee}$, induces the rank $3$ quadric
 $Q:H^0(C, K_C)^{\vee}\rightarrow H^0(C, K_C)$. The map $u\circ \tilde{Q}: \wedge^2 H^0(E)^{\vee}\rightarrow H^0(K_C)$ being surjective, its kernel $\mbox{Ker}(u\circ \tilde{Q})$ is a $7$-dimensional vector space containing the $5$-dimensional subspace $\mbox{Ker}(Q)$. We choose an arbitrary element $$[\bar{v}:=v+\mathrm{Ker}(Q)]\in \PP\Bigl(\frac{\mathrm{Ker}(u\circ \tilde{Q})}{\mathrm{Ker}(Q)}\Bigr)=\PP^1,$$ inducing a subspace
 $H^0(C, K_C)^{\vee}\subset \Lambda:=H^0(C, K_C)^{\vee}+\mathbb C v \subset \wedge^2 H^0(C, E)^{\vee},$
 with the property that $\mbox{Ker}(Q_{\Lambda})=\mbox{Ker}(Q)$, where $Q_{\Lambda}: \Lambda \rightarrow \Lambda^{\vee}$ is induced from $\tilde{Q}$ by restriction and projection. It follows that $\mbox{rk}(Q_{\Lambda})=4$ and there is a pencil of $8$-planes $\Lambda\supset \PP^7$ with this property.
\end{proof}

\vskip 4pt

 Let
$C \subset Q \subset \PP^7$
be a general canonical curve endowed with a vanishing theta-null, where $Q\in H^0\bigl(\PP^7, I_{C/\PP^7}(2)\bigr)$ is the corresponding rank $3$ quadric. We choose a general $8$-plane $\PP^7\subset \Lambda\subset \PP^{14}$ such that
$S:= \Lambda \cap \GG$
is a smooth K3 surface, and the lift of $Q$ to $\Lambda$
$$Q_{\Lambda}\in H^0\bigl(\Lambda, \mathcal{I}_{S/\Lambda}(2)\bigr)$$
has rank $4$ (cf. Proposition \ref{exprk4}). Moreover, we can assume that $S\cap \mbox{Sing}(Q_{\Lambda})=\emptyset$. The linear projection $f_{\Lambda}:\Lambda \dashrightarrow \PP^3$ with center $\mbox{Sing}(Q_{\Lambda})$, induces a regular map $f:S\rightarrow \PP^3$ with image the smooth quadric $Q_0\subset \PP^3$.  Then $S$ is endowed with two elliptic pencils $|C_1|$ and $|C_2|$ corresponding to the projections of $Q_0\cong \PP^1\times \PP^1$ onto the two factors.  Since $C\in |\OO_S(1)|$, one has a linear equivalence
$C\equiv C_1+C_2$, on $S$.
As already pointed out, $\mbox{deg}(f)=C_1\cdot C_2=C^2/2=7$. The condition $\mbox{rk}(Q_{\Lambda}\cap \PP^7)=\mbox{rk}(Q_{\Lambda})-1$, implies that the hyperplane $\PP^7\in (\Lambda)^{\vee}$ is the pull-back of a hyperplane from $\PP^3$, that is, $\PP^7=f^{-1}_{\Lambda}(\Pi_0)$, where $\Pi_0\in (\PP^3)^{\vee}$. This proves the following:

\begin{corollary}\label{explicitfibre}
The rational morphism $q:\cU\dashrightarrow \thet$ is dominant.
\end{corollary}
\begin{proof} Keeping the notation from above, if $[C]\in \thet$ is a general point corresponding to the rank $3$ quadric $Q\in H^0(\PP^7, \I_{C/\PP^7}(2))$, then $[S, \OO_S(C_1+C_2), C]\in q^{-1}([C]).$
\end{proof}

\vskip 4pt
We begin the proof of Proposition \ref{pencil8} while retaining the set-up described above. Let us choose a general line $l_0\subset \Pi_0$ and denote by $\{q_1, q_2\}:=l_0\cap Q_0$. We consider the pencil $\{\Pi_t\}_{t\in \PP^1}\subset (\PP^3)^{\vee}$ of planes through $l_0$ as well as the induced pencil of curves of genus $8$
$$\{C_t:=f^{-1}(\Pi_t)\subset S\}_{t\in \PP^1},$$
each endowed with a vanishing theta-null induced by the
pencil $f_{t}: C_t\rightarrow Q_0\cap \Pi_t$.
\vskip 5pt

This pencil contains precisely two \emph{reducible} curves, corresponding to the planes $\Pi_{1}, \Pi_{2}$ in $\PP^3$ spanned by the rulings of $Q_0$ passing through $q_1$ and $q_2$ respectively. Precisely, if $l_i, m_i\subset Q_0$ are the rulings passing through $q_i$ such that $l_1\cdot l_2=m_1\cdot m_2=0$, then it follows that for $\Pi_1=\langle l_1, m_2\rangle, \Pi_2=\langle l_2, m_1\rangle$, the fibres $f^{-1}(\Pi_{1})$ and $f^{-1}(\Pi_2)$ split into two elliptic curves $f^{-1}(l_i)$ and $f^{-1}(m_j)$ meeting transversally in $7$ points. The half-canonical $\mathfrak g^1_7$ specializes to a degree $7$ admissible covering
 $$f^{-1}(l_i)\cup f^{-1}(m_j) \stackrel{f}\rightarrow l_i\cup m_j, \  \ i\neq j,$$ such that the $7$ points in $f^{-1}(l_i)\cap f^{-1}(m_j)$ map to $l_i \cap m_j$. To determine the point in $\ss_8^+$ corresponding to the admissible covering $\bigl(f^{-1}(l_i)\cup f^{-1}(m_j),\  f_{| f^{-1}(l_i)\cup f^{-1}(m_j)}\bigr)$, one must insert $7$ exceptional components at all the points of intersection of the two components. We denote by $R\subset \thet\subset \ss_8^+$ the pencil of spin curves obtained via this construction.

\begin{lemma}\label{nodalpencil} Each member  $C_t\subset S$ in the above constructed pencil is nodal. Moreover, each curve $C_t$ different from
$f^{-1}(l_1)\cup f^{-1}(m_2)$ and $f^{-1}(l_2)\cup f^{-1}(m_1)$ is irreducible. It follows that $R\cdot \alpha_i=R\cdot \beta_i=0$ for $i=1, \ldots, 4$.
\end{lemma}
\begin{proof} This follows since $f:S\rightarrow Q_0$ is a regular morphism and the base line $l_0\subset H_0$ of the pencil $\{\Pi_t\}_{t\in \PP^1}$ is chosen to be general.
\end{proof}

\begin{lemma} $R \cdot \pi^*(\mm^2_{7,8})= 0$.
\end{lemma}
\begin{proof} We show instead that $\pi_*(R)\cdot \mm_{8, 7}^2=0$. From Lemma \ref{nodalpencil}, the curve $R$ is disjoint  from  the divisors $A_i, B_i$ for $i=1, \ldots, 4$, hence  $\pi_*(R)$ has the numerical characteristics of a Lefschetz pencil of curves of genus $8$ on a fixed $K3$ surface.

\noindent
In particular, $\pi_*(R)\cdot \delta/\pi_*(R)\cdot \lambda=6+12/(g+1)=s(\mm_{8, 7}^2)$ and $\pi_*(R)\cdot \delta_i=0$ for $i=1, \ldots, 4$. This implies the statement. \end{proof}
\vskip 4pt

\begin{lemma}\label{thetintersection}
 $R \cdot \thet = -1$. \end{lemma}
 \begin{proof} We have already determined that
$
R\cdot \lambda =\pi_*(R)\cdot \lambda= \chi(\tilde{S}, \OO_{\tilde S}) + g - 1 = 9,
$ where $\tilde{S}:=\mbox{Bl}_{2g-2}(S)$ is the blow-up of $S$ at the points $f^{-1}(q_1)\cup f^{-1}(q_2)$.
Moreover,
\begin{equation}\label{relation}
R\cdot \alpha_0 + 2R\cdot \beta_0 =\pi_*(R)\cdot \delta_0= c_2(\tilde X) + 4(g-1) = 38 + 28 = 66.
\end{equation}
To determine $R\cdot \beta_0$ we study the local structure of $\ss_8^+$ in a neighbourhood of one of the two points, say $t^*\in R$ corresponding to a reducible curve, say $f^{-1}(l_1)\cup f^{-1}(m_2)$, the situation for $f^{-1}(l_2)\cup f^{-1}(m_1)$ being of course identical.
We set $\{p\}:=l_1\cap m_2\in Q_0$ and $\{x_1, \ldots, x_7\}:=f^{-1}(p)\subset S$. We insert exceptional components $E_1, \ldots, E_7$ at the nodes $x_1, \ldots, x_7$ of $f^{-1}(l_1)\cup f^{-1}(m_2)$ and denote by $X$ the resulting quasi-stable curve. If $$\mu: f^{-1}(l_1)\cup f^{-1}(m_2)\cup E_1\cup \ldots \cup E_7\rightarrow f^{-1}(l_1)\cup f^{-1}(m_2)$$ is the stabilization morphism, we set $\{y_i, z_i\}:=\mu^{-1}(x_i)$, where $y_i\in E_i\cap f^{-1}(l_1)$ and $z_i\in E_i\cap f^{-1}(m_2)$ for $i=1, \ldots, 7$. If $t^*=[X, \eta, \beta]$, then $\eta_{f^{-1}(l_1)}=\OO_{f^{-1}(l_1)}, \ \eta_{f^{-1}(m_2)}=\OO_{f^{-1}(m_2)}$, and of course $\eta_{E_i}=\OO_{E_i}(1)$. Moreover, one computes that $\mbox{Aut}(X, \eta, \beta)=\mathbb Z_2$, see \cite{C} Lemma 2.2, while clearly $\mbox{Aut}(f^{-1}(l_1)\cup f^{-1}(m_2))=\{\mbox{Id}\}$.

\vskip 5pt
If $\mathbb C_{\tau}^{3g-3}$  denotes the versal deformation space of $[X, \eta, \beta]\in \ss_g^+$, then there are local parameters $(\tau_1, \ldots, \tau_{3g-3})$, such that for $i=1, \ldots, 7$, the locus $\bigl(\tau_i=0\bigr)\subset \mathbb C_{\tau}^{3g-3}$ parameterizes spin curves for which the exceptional component $E_i$ persists. It particular, the pull-back $\mathbb C_{\tau}^{3g-3}\times _{\ss_g^+} B_0$ of the boundary divisor $B_0\subset \ss_g^+$ is given by the equation $\bigl(\tau_1\cdots \tau_7=0\bigr)\subset \mathbb C_{\tau}^{3g-3}$. The group $\mbox{Aut}(X, \eta, \beta)$ acts on $\mathbb C_{\tau}^{3g-3}$ by
$$(\tau_1, \ldots, \tau_7, \tau_8, \ldots, \tau_{3g-3})\mapsto (-\tau_1, \ldots, -\tau_7, \tau_8, \ldots, \tau_{3g-3}),$$
and since an \'etale neighbourhood of $t^*\in \ss_g^+$ is isomorphic to $\mathbb C_{\tau}^{3g-3}/\mbox{Aut}(X, \eta, \beta)$, we find that the Weil divisor $B_0$ is not Cartier around $t^*$ (though $2B_0$ is Cartier). It follows that the intersection multiplicity of $R\times _{\ss_g^+} \mathbb C_{\tau}^{3g-3}$ with the locus $(\tau_1\cdots \tau_7)=0$ equals $7$, that is, the intersection multiplicity of $R\cap \beta_0$ at the point $t^*$ equals $7/2$, hence
$$
R\cdot \beta_0 =\bigl(R\cdot \beta_0\bigr)_{f^{-1}(l_1)\cup f^{-1}(m_2)}+\bigl(R\cdot \beta_0\bigr)_{f^{-1}(l_2)\cap f^{-1}(m_1)}=\frac{7}{2}+\frac{7}{2}=7.
$$
Then using (\ref{relation}) we find that $R\cdot \alpha_0=66-14=52$, and finally
$$
R\cdot \thet = \frac{1}{4}R\cdot \lambda  - \frac{1}{16}R\cdot \alpha_0 = \frac{9}{4}- \frac{52}{16}= -1.
$$
\end{proof}


\begin{thebibliography}{EMS}
\bibitem[AF]{AF} M. Aprodu and G. Farkas, {\em{Green's Conjecture for general covers}},  arXiv:1105.3933, to appear in Vector bundles and compact moduli (Athens Georgia 2010), Contemporary Mathematics.
\bibitem[BM]{BM} A. Beauville and Y. Merindol, {\em{Sections hyperplanes des surfaces $K3$}}, Duke Math. Journal \textbf{4} (1987), 873-878.
\bibitem[Ca]{Ca} F. Catanese, {\em{On the rationality of certain moduli spaces
related to curves of genus $4$}}, Springer Lecture Notes in Mathematics \textbf{1008} (1983), 30-50.
\bibitem[Cor]{C} M. Cornalba, {\em{Moduli of curves and theta-characteristics}},
in: Lectures on Riemann surfaces (Trieste, 1987), 560-589.
\bibitem[CD]{CD} F. Cossec and I. Dolgachev, {\em{Enriques surfaces}}, Progress in Mathematics Vol. 76,  1989, Birkh\"auser.
\bibitem[Do1]{Do} I. Dolgachev, {\em{Mirror symmetry for lattice polarized $K3$ surfaces}}, J. Math. Sciences \textbf{81} (1996), 2599-2630.
\bibitem[Do2]{Do2} I. Dolgachev, {\em{Rationality of $\cR_2$ and $\cR_3$}}, Pure and Applied Math. Quarterly, \textbf{4} (2008), 501-508.
\bibitem[Do]{D} R. Donagi, {\em{The unirationality of $\mathcal{A}_5$}},
Annals of Mathematics \textbf{119} (1984), 269-307.
\bibitem[EH1]{EH1} D. Eisenbud and J. Harris, \ {\em{The Kodaira dimension  of the
moduli space of curves of genus $23$ }}, Inventiones Math.
\textbf{90} (1987), 359--387.
\bibitem[F]{F1} G. Farkas, {\em{The birational type of the moduli space of even spin curves}}, Advances in Mathematics 223 (2010), 433-443.
\bibitem[FL]{FL} G. Farkas and K. Ludwig, {\em{The Kodaira dimension of the moduli
space of Prym varieties}}, Journal of the European Mathematical Society \textbf{12} (2010), 755-795.
\bibitem[FP]{FP} G. Farkas and M. Popa, {\em{Effective divisors on
$\mm_g$, curves on $K3$ surfaces and the Slope Conjecture}}, Journal of Algebraic Geometry \textbf{14} (2005), 151-174.
\bibitem[FV]{FV} G. Farkas and A. Verra, {\em{The geometry of the moduli space of odd spin curves}}, arXiv:1004.0278.
\bibitem[GS]{GS} A. Garbagnati and A. Sarti, {\em{Projective models of $K3$ surfaces with an even set}}, Advances in Geometry \textbf{8} (2008), 413-440.
\bibitem[vGS]{vGS} B. van Geemen and A. Sarti, {\em{Nikulin involutions on $K3$ surfaces}}, Mathematische Zeitschrift \textbf{255} (2007), 731-753.
\bibitem[GL1]{GL} M. Green and R. Lazarsfeld, {\em{On the projective normality of complete linear series on an algebraic curve}}, Inventiones Math. \textbf{83} (1986), 73-90.
\bibitem[GL2]{GL2} M. Green and R. Lazarsfeld, {\em{Special divisors on curves on a $K3$ surface}}, Inventiones Math. \textbf{89} (1987), 357-370.
\bibitem[HM]{HM} J. Harris and D. Mumford, {\em{On the Kodaira
dimension of $\mm_g$}}, Inventiones Math. \textbf{67} (1982), 23-88.
\bibitem[HMo]{HMo} J. Harris and I. Morrison, {\em{Slopes of
effective divisors on the moduli space of stable curves}}, Inventiones
Math. \textbf{99} (1990), 321-355.
\bibitem[HP]{HP} W.V.D. Hodge and D. Pedoe, {\em{Methods of algebraic geometry}}, Volume Two, Cambridge University Press, Reissued in 1994.
\bibitem[ILS]{ILS} E. Izadi, M. Lo Giudice and G. Sankaran, {\em{The
moduli space of \'etale double covers of genus $5$ is unirational}}, Pacific Journal of Mathematics \textbf{239} (2009), 39-52.
\bibitem[Laz]{Laz} R. Lazarsfeld, {\em{Brill-Noether-Petri without degenerations}}, Journal of Differential Geometry
\textbf{23} (1986), 299-307.
\bibitem[Mo]{Mo} D. Morrison, {\em{On $K3$ surfaces with large Picard number}}, Inventiones Math. \textbf{75} (1984), 105-121.
\bibitem[M1]{M1} S. Mukai, {\em{Curves, $K3$ surfaces and Fano $3$-folds of genus $\leq 10$}}, in: Algebraic geometry and commutative algebra
357-377, 1988 Kinokuniya, Tokyo.
\bibitem[M2]{M2} S. Mukai, {\em{Curves and Grassmannians}}, in: Algebraic Geometry and Related Topics, eds. J.-H. Yang, Y. Namikawa, K. Ueno, 19-40, 1992   International Press.
\bibitem[M3]{M3} S. Mukai, {\em{Curves and $K3$ surfaces of genus eleven}}, in: Moduli of vector bundles, Lecture Notes in Pure and Applied Mathematics Vol. 179, Dekker (1996), 189-197.
\bibitem[M4]{M4} S. Mukai, {\em{Biregular classification of Fano $3$-folds and Fano manifolds of coindex $3$}}, Proceedings of the  National Academy Sciences USA, Vol. 86, 3000-3002, 1989.
\bibitem[Ni]{Ni} V.V. Nikulin, {\em{Kummer surfaces}}, Izvestia Akad. Nauk SSSR \textbf{39} (1975), 278-293.
\bibitem[SD]{SD} B. Saint-Donat, {\em{Projective models of $K3$ surfaces}}, American Journal of Mathematics \textbf{96} (1974), 602-639.
\bibitem[V1]{V1} A. Verra, {\emph{A short proof of the unirationality
of $\cA_5$}}, Indagationes Math. \textbf{46} (1984), 339-355.
\bibitem[V2]{V2} A. Verra, {\em{On the universal principally polarized abelian variety of dimension $4$}}, in: Curves and abelian varieties (Athens, Georgia, 2007)  Contemporary Mathematics Vol. 345, 2008, 253-274.
\bibitem[Vo]{Vo} C. Voisin, {\em{Green's generic syzygy conjecture
for curves of even genus lying on a $K3$ surface}}, Journal of
European Mathematical Society \textbf{4} (2002), 363-404.
\bibitem[We]{W} G. Welters,
\emph{A theorem of Gieseker-Petri type for Prym varieties},
Annales Scientifique \'Ecole  Normale Sup\'erieure \textbf{18} (1985), 671--683.

\end{thebibliography}
\end{document}